\documentclass[11pt,a4paper,reqno]{amsart}

\usepackage[
    a4paper,
    left=2.6cm,
    right=2.6cm,
    top=2.9cm,
    bottom=2.9cm,
    headheight=14pt,
    headsep=0.5cm
]{geometry}

\usepackage[T1]{fontenc}
\usepackage[utf8]{inputenc}
\usepackage{amsmath,amssymb,amscd,mathtools}

\usepackage{tikz-cd}
\usepackage{enumitem}
\usepackage{hyperref}
\usepackage[nameinlink,capitalise]{cleveref}
\usepackage{fancyhdr}
\usepackage{aliascnt}
\usepackage[nameinlink,capitalise]{cleveref}
\hypersetup{
  colorlinks=true,
  linkcolor=blue,
  citecolor=blue,
  urlcolor=blue
}

\numberwithin{equation}{section}

\theoremstyle{plain}
\newtheorem{theorem}{Theorem}[section]

\newaliascnt{corollary}{theorem}
\newtheorem{corollary}[corollary]{Corollary}
\aliascntresetthe{corollary}

\newaliascnt{proposition}{theorem}
\newtheorem{proposition}[proposition]{Proposition}
\aliascntresetthe{proposition}

\newaliascnt{lemma}{theorem}
\newtheorem{lemma}[lemma]{Lemma}
\aliascntresetthe{lemma}

\theoremstyle{definition}

\newaliascnt{definition}{theorem}
\newtheorem{definition}[definition]{Definition}
\aliascntresetthe{definition}

\newaliascnt{remark}{theorem}
\newtheorem{remark}[remark]{Remark}
\aliascntresetthe{remark}

\newaliascnt{example}{theorem}

\aliascntresetthe{example}

\crefname{theorem}{Theorem}{Theorems}
\Crefname{theorem}{Theorem}{Theorems}

\crefname{corollary}{Corollary}{Corollaries}
\Crefname{corollary}{Corollary}{Corollaries}

\crefname{proposition}{Proposition}{Propositions}
\Crefname{proposition}{Proposition}{Propositions}

\crefname{lemma}{Lemma}{Lemmas}
\Crefname{lemma}{Lemma}{Lemmas}

\crefname{definition}{Definition}{Definitions}
\Crefname{definition}{Definition}{Definitions}

\crefname{remark}{Remark}{Remarks}
\Crefname{remark}{Remark}{Remarks}

\crefname{example}{Example}{Examples}
\Crefname{example}{Example}{Examples}

\crefname{equation}{Equation}{Equations}
\Crefname{equation}{Equation}{Equations}

\usepackage{booktabs}                   
\usepackage{graphicx}
\usepackage{graphicx}
\usepackage{makecell}

\newcommand{\F}{\mathbb F}
\newcommand{\Z}{\mathbb Z}
\newcommand{\Gal}{\operatorname{Gal}}

\newcommand{\im}{\operatorname{im}}

\newcommand{\Hom}{\operatorname{Hom}}

\newcommand{\res}{\operatorname{res}}
\newcommand{\infmap}{\operatorname{inf}}

\newcommand{\floor}[1]{\left\lfloor #1\right\rfloor}

\newcommand{\gp}{\mathfrak p}
\newcommand{\gP}{\mathfrak P}
\newcommand{\gm}{\mathfrak m}

\newcommand{\Fp}{\mathbb F_p}

\newcommand{\nr}{\mathrm{nr}}

\newcommand{\tg}{\operatorname{tg}}

\begin{document}

\title[Ihara Constants from Zassenhaus Filtration]{Zassenhaus Filtrations and Golod-Shafarevich Methods for Ihara Constants}

\author{Liming Ma}
\address{School of Mathematical Sciences,
University of Science and Technology of China,
Hefei 230026, China}
\email{lmma20@ustc.edu.cn}

\author{Chaoping Xing}
\address{School of Computer Sciences,
Shanghai Jiao Tong University,
Shanghai 200240, China}
\email{xingcp@sjtu.edu.cn}

\author{Zugan Xing}
\address{School of Mathematical Sciences,
Shanghai Jiao Tong University,
800 Dongchuan Road,
Shanghai 200240, China}
\email{xingzugan@aliyun.com}



\begin{abstract}
Let \(X\) be a curve over \(\mathbb{F}_q\). Denote by \(N(X)\) and \(g(X)\) the number of \(\mathbb{F}_q\)-rational points and the genus of \(X\) respectively. The Ihara constant
\begin{equation*}
A(q)=\limsup_{g(X)\to\infty}\frac{N(X)}{g(X)}
\end{equation*}
measures the asymptotic abundance of rational points on curves over \(\mathbb{F}_q\). It is a central problem in the study of algebraic geometry and has also attracted broader mathematical interest. In particular, the constant $A(2)$ is discussed in the \emph{Optimization Constants in Mathematics} collection maintained by Terence Tao et al \cite[Number 33]{DIT26}.

In this paper, we study the Zassenhaus filtration of pro-\(p\) groups arising from class field towers. By combining Golod--Shafarevich theorem, embedding problems and central class field theory, we establish a new criterion for the infinitude of class field towers.
As an application, we obtain an improvement on the best known lower bounds of Ihara constant $A(2)$ and $A(3)$ which are explicitly given by
\[
  A(2)\geq 0.317779\cdots \text{ and }
  A(3)\geq 0.496074\cdots.
\]

\end{abstract}

\maketitle

\section{Introduction and main statements}
Let \(p\) be a prime, \(q=p^e\) be a prime power and $\F_q$ be the finite field with $q$ elements. For a projective, nonsingular, geometrically irreducible curve \(X/\mathbb{F}_q\) of genus \(g\), denote by \(N(X)\) the number of \(\mathbb{F}_q\)-rational points on \(X\). The Hasse--Weil bound gives
\[
   |N(X)-(q+1)|\le 2g\sqrt q.
\]
However, even when \(q\) is a square, Ihara \cite{Ihara1981} showed that the Hasse--Weil bound cannot be attained for curves of sufficiently large genus.

For any fixed \(q\), let $N_q(g)$ be the maximum number of rational places of curves with genus $g$, that is,
\[ N_q(g)=\max\{N(X): X/\mathbb{F}_q \text{ is a curve with genus }g\}. \]
In \cite{Ihara1981}, Ihara introduced the following quantity
\[A(q)=\limsup_{g\to\infty}\frac{N_q(g)}{g}\]
which is called the Ihara constant, and proved that $A(q)\ge \sqrt{q}-1$ for any square $q$ from deep results in algebraic geometry and modular curves. Independently, Tsfasman, Vl\u{a}du\c{t} and Zink proved that $A(p^2)\ge p-1$ from modular curves in \cite{Tsfasman1982}.
The Drinfeld--Vl\u{a}du\c{t} bound \cite{Drinfeld1983} states $A(q)\leq \sqrt{q}-1$, which is still the best upper bound.
Thus, when \(q\) is a square, we have $$A(q)=\sqrt{q}-1.$$
For any non-square $q$, the exact value of $A(q)$ is still unknown.

The Ihara constant is an important asymptotic quantity in number theory and algebraic geometry. The main open problem of determining its exact value and obtaining better bounds has attracted broader mathematical interest. In particular, $A(2)$ is discussed in the \emph{Optimization Constants in Mathematics} collection maintained by Terence Tao et al.~\cite[Number 33]{DIT26}.

One of the main motivations of Ihara constant $A(q)$ comes from coding theory and cryptography \cite{CCCX09,Xing2014,CC06,LWX23,HXbook2009,Tsfasman1982}.
In particular, if $q\ge 49$ is a square, then there exists a family of asymptotically good algebraic geometry codes with parameters exceeding the Gilbert--Varshamov bound \cite{Tsfasman1982}. For applications in coding theory and cryptography, it is essential to have explicit towers of function fields.
Garcia and Stichtenoth provided two explicit optimal towers of function fields attaining the  Drinfeld--Vl\u{a}du\c{t} upper bound recursively defined by Artin--Schreier curves over finite field $\F_q$ with $q=p^{2m}$ \cite{GS95,GS96}.  Bassa {\it et al.} provided an explicit tower of function fields over $\F_q$ with $q=p^{2m+1}$ for $m\ge 1$ such that
$$A(p^{2m+1})\ge \frac{2(p^{m+1}-1)}{p+1+\zeta}\quad \text{ with } \quad \zeta=\frac{p-1}{p^m-1}$$
in \cite{BBGS15}.
In $2020$, Bassa and Ritzenthaler gave an explicit construction for good recursive towers with $A(q)\ge 2/(q-2)>0$ over any finite field $\F_q$ with $q\ge 4$ including prime finite fields \cite{BR20}.

Serre initially used the class field tower method to obtain the logarithmic lower bound $A(q)\geq \log_2(q)/96$ for any prime power \(q\) \cite{Serrebook}. In particular, if \(q\) is a small prime, then the logarithmic lower bound can be improved. In \cite{Se83}, Serre proved that
\[
A(2)\geq \frac{2}{9}=0.2222\cdots.
\]
Schoof rediscovered this fact by exhibiting a different class field tower in \cite{Sc92}. Niederreiter and Xing \cite{HX1998} used unramified class field towers and cyclotomic function fields to show that
$$A(2)\geq 81/317=0.2555\cdots \quad\text { and } \quad A(3)\geq 62/163=0.38037\cdots.$$
In \cite{XY07}, Xing and Yeo improved that $A(2)\ge 97/376=0.257979\cdots$ by cyclotomic function fields.
On the other hand, Maire developed tamely ramified towers in both the number field setting \cite{HMcompositio} and the function field setting \cite{HM00,BM02}. In particular, the Ihara constant $A(3)$ can be improved by
\[
A(3)\geq 12/25=0.48.
\]

In 2002, Kuhnt developed unramified class field towers over base extensions with bounded wild ramification and showed $A(2)\geq 39/129= 0.302325\cdots$ in his dissertation \cite{Kuhnt}.
 In 2013, Duursma and Mak \cite{Iharacompositio} summarized the Kuhnt's idea and obtained the improved lower bounds
\[
A(2)\geq 0.316999\cdots \quad \text{ and } \quad
A(3)\geq 0.492876\cdots.
\]
Moreover, we note a gap in  \cite[Proposition 5.6]{Iharacompositio}, which may lead to a smaller lower bound for $A(3)$ than $ 0.492876\cdots$.  For details, see \cref{remark:A3}.

In this paper, we refine this approach by introducing the Zassenhaus filtration of pro-\(p\) groups arising from class field towers to obtain the new lower bounds of Ihara constant $A(q)$ for $q=2$ and $3$.

\begin{theorem}[Main tower criterion]
\label{thm:maincriterion}
Let $k/\mathbb{F}_q$ be a function field of genus $g$, where $q=p^e$. Let $S$ and $T$ be disjoint nonempty finite sets of places of $k$, with $T$ consisting of rational places. Put $t=|T|>0.$ For each $\mathfrak{p}\in S$, let $f_{\mathfrak{p}}=\deg(\mathfrak{p}),$ and choose a positive integer $\nu_{\mathfrak{p}}$. Define $\mathfrak{m}=\sum_{\mathfrak{p}\in S} \nu_{\mathfrak{p}}\mathfrak{p}.$ Let $K=k_{\mathfrak{m}}^{T}$ be the maximal elementary abelian $p$-extension of $k$ with conductor bounded by $\mathfrak{m}$, in which every place in $T$ splits completely. Set $n_{\mathfrak{p}}=ef_{\mathfrak p}
\left(
\nu_{\mathfrak p}-1-
\left\lfloor
\frac{\nu_{\mathfrak p}-1}{p}
\right\rfloor
\right)$ and
\[
d=1+\sum_{\mathfrak{p}\in S}e f_{\mathfrak{p}}\left(\nu_{\mathfrak{p}}-1-\left\lfloor\frac{\nu_{\mathfrak{p}}-1}{p}\right\rfloor\right)-t.
\]
Assume that $d\geq 1$. If there exists $\tau\in(0,1)$ such that
\begin{align}
1-d\tau
+
\left(
d+t-1+
\sum_{\mathfrak{p}\in S}
\binom{n_{\mathfrak{p}}}{2}
\right)\tau^2
+
\sum_{\mathfrak{p}\in S}
n_{\mathfrak{p}}\tau^p
<0,
\label{eq:main-inequality}
\end{align}
then the $(T,p)$-class field tower above $K$ is infinite. Moreover, we have
\[
A(q)
\geq
\frac{t}{
\displaystyle
g-1+
({1}/{2[K:k]})
\sum_{\chi}
\deg\mathfrak{f}_{\chi}
}\ge \frac{t}{
\displaystyle
g-1+
({1}/{2})
\sum_{\mathfrak{p}\in S}
f_{\mathfrak{p}} \nu_{\mathfrak{p}}
},
\]
where $\chi$ runs through all characters of $\operatorname{Gal}(K/k)$ and $\mathfrak{f}_{\chi}$ denotes the Artin conductor of $\chi$.
\end{theorem}

Using \cref{thm:maincriterion}, we can construct several function fields $K$ with many rational places relative to their genera and improve the lower bounds of Ihara constant $A(q)$ for $q=2$ and $3$.

\begin{theorem}
\label{thm:A2}
 The Ihara constants $A(2)$ and  $A(3)$ are lower bounded by   $$A(2)\geq 0.317779\cdots \quad \text{ and } \quad A(3)\geq 0.496074\cdots.$$
\end{theorem}

The applications of the Zassenhaus filtration to class field tower problems over number fields \cite{Vogel2005,HMR2020,AhlqvistCarlson2025,HMRcutt} motivate us to develop a corresponding framework for function fields. To our knowledge, this is the first application of the Zassenhaus filtration to the construction of function fields over finite fields with many rational places. It would be interesting to combine this framework with Massey products to further improve lower bounds for Ihara constants or to construct function fields over finite fields with large automorphism groups \cite[Proposition 5.1]{SXX26}.

This article is organized as follows.
In Section \ref{sec:2}, we introduce the Zassenhaus filtration of pro-$p$ groups and Golod--Shafarevich theorem on infinite tower of function fields.
In Section \ref{sec:3}, we estimate the generator rank of the Galois group \(G_{S,\nu}^{T}\) of the maximal pro-\(p\) extension \(K_{S,\nu}^{T}\) of \(K\) which is unramified outside \(S\), has ramification depth at most \(\nu_{\mathfrak p}\) at \(\mathfrak p\in S\), and in which every prime of  \(T\) splits completely.
In Section \ref{sec:relation-detailed}, we estimate the upper bound of the difference of relation rank and generator rank.
In Section \ref{sec:5}, we provide the complete proof of the main tower criterion \cref{thm:maincriterion}.
In Section \ref{sec:6}, we give examples of function fields $k/\mathbb{F}_q$ for which the criterion of \cite[Theorem 1.1]{Iharacompositio} does not prove the infinitude of the corresponding class field towers, whereas our criterion does. Moreover, a detailed comparison with the method given in \cite{Iharacompositio} is provided in Section \ref{subsec:6.2}.

\section{Zassenhaus filtration and class field tower}\label{sec:2}

First, let us summarize the method of class field tower introduced by Serre to obtain the lower bound of Ihara constant.
Let \(p\) be a prime, \(q=p^e\) be a prime power and $\F_q$ be the finite field with $q$ elements.
Let $K/\F_q$ be a global function field of genus $g(K)$, and let $T$ be a nonempty finite set of rational places of $K$.  For a prime $\ell$, the $(T,\ell)$-Hilbert class field $H(K)$ is the maximal unramified abelian $\ell$-extension of $K$ in which all places of $T$ split completely.  Iterating this construction gives
\[
   K=K_0\subseteq K_1\subseteq K_2\subseteq \cdots,
   \qquad K_{n+1}=H(K_n).
\]
If this tower is infinite, then all places above $T$ split completely at every level, while the extensions are unramified.  Thus, the Hurwitz genus formula yields
\begin{equation}
\label{eq:basic-bound}
   A(q)\ge \frac{|T|}{g(K)-1}
\end{equation}
from \cite[Proposition 2.2]{Sc92} or \cite[Theorem 2.7.6]{HXrationalbook}.

\begin{theorem}
\cite[Theorem 3.4]{GSgroup-survey}
\label{thm:GS}
Assume that $G$ is a finitely generated pro-$p$ group.
Let $d_p(G)$ be the generator rank of $G$ and $r_p(G)$ be the relation rank of $G$.
If there exists a real number \(\tau\in(0,1)\) such that
\begin{align*}
1-d_p(G)\tau+ r_p(G)\tau^2<0,
\label{eq:main-inequality}
\end{align*}
then $G$ is infinite.
\end{theorem}

The usual class field tower argument applies this theorem to $\Gal(L/K)$.  The refinement used here follows Kuhnt \cite{Kuhnt}: one writes $L/K/k$, where $K/k$ is a finite Galois $p$-extension with controlled wild ramification and $L/K$ is unramified, and applies generator and relation estimates to the larger group
\[
   G=\Gal(L/k).
\]
In \cite[Proposition 5.1 and Theorem 4.1]{Iharacompositio}, an upper bound for the difference between the generator rank and the relation rank is obtained in terms of the relation rank of a finite abelian group. This bound is attained when the group is elementary abelian.

On the other hand, when the preceding Golod--Shafarevich theorem is formulated in the pro-\(p\) setting; see \cref{def:GS-grp}, the defining relations of an elementary abelian \(p\)-group are naturally graded by their Zassenhaus degrees; see \cref{def:Zass-degree}.  \cref{Defining relation} and \cref{prop:abelian-standard-relations}. More precisely, these relations contain a degree-\(2\) component and a higher degree component. The estimate in the theorem above takes into account only the contribution from relations of Zassenhaus degree \(2\). Consequently, it may be weaker than the pro-\(p\) Golod--Shafarevich formulation in \cref{def:GS-grp} and \cref{Defining relation}, which retains information from relations in all Zassenhaus degrees. This observation motivates a more detailed study of the Zassenhaus filtration of pro-\(p\) groups from \cite{GSduke,GSgroup-survey}.

\subsection{Golod-Shafarevich group and Zassenhaus filtration}
Fix a prime number $p$, and let $G$ be a finitely generated pro-$p$ group. Recall that
\[
G=\varprojlim_{N\in\Omega_p(G)} G/N,
\]
where $\Omega_p(G)$ is the set of open normal subgroups of $G$ (all of which have $p$-power index). We shall be interested in the completed group algebra $\mathbb F_p[[G]]$ which is defined as the corresponding inverse limit of $\mathbb F_p$-group algebras:
\[
\mathbb F_p[[G]]=\varprojlim_{N\in\Omega_p(G)} \mathbb F_p[G/N].
\]
Suppose that $G$ is given in the form $G=F/R$,  where $F$ is a finitely generated free pro-$p$ group, $\mathcal{R}$ is a subset of $F$ and $ R $ is the closed normal subgroup of $F$ generated by $\mathcal{R}$. By \cite[Theorem 7.11]{Kochbook}, there is a natural isomorphism
\[
\mathbb F_p[[G]]\cong \mathbb F_p[[F]]/I_\mathcal{R},
\]
where $I_\mathcal{R}$ is the closed ideal of $\mathbb F_p[[F]]$ generated by the set $\{r-1:r\in \mathcal{R}\}$.

Let $X=\{x_1,\ldots,x_d\}$ be a free generating set of $F$. By a theorem of Lazard \cite[Theorem 7.16]{Kochbook}, the completed group algebra $\mathbb F_p[[F]]$ is isomorphic to the algebra of formal power series  $\mathbb F_p\langle\!\langle u_1,\ldots,u_d\rangle\!\rangle$ in non-commuting variables, under the map $x_i\mapsto 1+u_i$ for $1\le i\le d$. Note that this map yields an embedding of $F$ into  the multiplicative group $\mathbb F_p\langle\!\langle u_1,\ldots,u_d\rangle\!\rangle^{\times}$ of $\mathbb F_p\langle\!\langle u_1,\ldots,u_d\rangle\!\rangle$. This embedding is called the Magnus embedding.

\begin{definition}
\label{def:Zass-degree}
Let $X=\{x_1,\ldots,x_d\}$ and $U=\{u_1,\ldots,u_d\}$ be finite sets of the same cardinality. Let $F$ be the free pro-$p$ group on $X$, and let $\iota:F\to \mathbb F_p\langle\!\langle U\rangle\!\rangle^{\times}$ be the Magnus embedding. Define the degree function $D:F\to \mathbb N\cup\{\infty\}$ by
\[
D(f)=\deg(\iota(f)-1),
\]
where $\deg$ is the usual degree of a power series in the variables $u_1,\ldots,u_d$ and assume $D(1)=\infty$.
\end{definition}

\begin{definition}
\label{def:GS-grp}
\begin{enumerate}[label=\textup{(\roman*)}, leftmargin=2.5em]
\item A (pro-$p$) presentation $(X,\mathcal{R})$ is said to satisfy the Golod-Shafarevich (GS for short) condition if there exists $\tau\in(0,1)$ such that $P_G(\tau)=1-|X|\tau+H_\mathcal{R}(\tau)<0$, where $H_\mathcal{R}(\tau)=\sum_{r\in \mathcal{R}} \tau^{D(r)}$ and $P_G(\tau)$ is called the Golod-Shafarevich polynomial.
\item A pro-$p$ group $G$ is called a Golod-Shafarevich group if it has a presentation satisfying the GS condition.
\end{enumerate}
\end{definition}
\begin{theorem}[Golod-Shafarevich]
\label{thm:GS-infinite}\cite[Theorem 3.1]{GSgroup-survey}
Golod-Shafarevich groups are infinite.
\end{theorem}

Now we introduce the Zassenhaus filtration of G.

\begin{definition}
Let $G$ be a finitely generated pro-$p$ group. Let $M$ be the augmentation ideal of the completed group algebra $\mathbb F_p[[G]]$, that is, $M$ is the ideal generated by the set $\{g-1:g\in G\}$. For each $n\in\mathbb N$, let $D_nG=\{g\in G:g-1\in M^n\}$. The series $\{D_nG\}$ is called the Zassenhaus filtration of $G$.
\end{definition}

\begin{remark}
    If $F$ is a finitely generated free pro-$p$ group and $f\in F\setminus\{1\}$, it is easy to see that $D(f)=n$ if and only if $f\in D_nF\setminus D_{n+1}F$. In particular, this shows that the function $D$ does not depend on the choice of a free generating set $X$ of $F$.
\end{remark}

Let $\gamma_iG$ denote the $i$-th term of the lower central series of $G$.
Thus the lower central series is defined recursively by $\gamma_1G=G$ and
\[
\gamma_{i+1}G=[\gamma_iG,G]
\]
for $i\geq 1$, here $[\gamma_iG,G]$ is the commutator group of $\gamma_iG$ and $G$.

\begin{proposition}\label{Zassenhau filtration}
\cite[Proposition 3.2]{GSgroup-survey}
The Zassenhaus filtration $\{D_nG\}$ can alternatively be described by
\[
D_nG=\prod_{i\cdot p^j\ge n}(\gamma_iG)^{p^j}.
\]
\end{proposition}

\subsection{Zassenhaus filtration on the relation module}
 In this article, we assume that $$ \dim_{\Fp}H^2(G,\Fp)<\infty.$$For convenience, subgroups generated by subsets of pro-$p$ groups are understood to be topologically generated and hence closed. In this section, we always assume that the finitely generated pro-$p$ group G admits a minimal pro-\(p\) presentation as follows:
\[
        1\to R\to F\to G\to1,
\]
where \(F\) is a free pro-\(p\) group.  We refer to $R/R^p[R,F]$ as the \emph{relation module} of \(G\) with respect to the chosen minimal presentation. The Pontryagin dual of the relation module $(R/R^p[R,F])^\vee$ can be identified with the continuous $ \mathbb F_p$-linear dual of the relation module $\Hom_{\mathrm{cont}}\bigl(R/R^p[R,F],\Fp\bigr)$ naturally. Moreover, we have the following duality result.
\begin{lemma}
\label{lem:relation mod dual}
One has
\[
        H^2(G,\Fp)\simeq (R/R^p[R,F])^\vee.
\]
In particular, the relation rank of $G$ is
\[
        r_p(G)=\dim_{\Fp}H^2(G,\Fp)
        =\dim_{\Fp}R/R^p[R,F].
\]
\end{lemma}

\begin{proof}
The Hochschild--Serre five-term exact sequence gives
\[
0\to H^1(G,\Fp)\to H^1(F,\Fp)\to H^1(R,\Fp)^G
\to H^2(G,\Fp)\to H^2(F,\Fp).
\]
Since \(F\) is free pro-\(p\), we have $H^2(F,\Fp)=0.$ Since the presentation is minimal, the map $H^1(G,\Fp)\to H^1(F,\Fp)$ is an isomorphism.  Hence, we have $H^2(G,\Fp)\simeq H^1(R,\Fp)^G.$ Now $H^1(R,\Fp)=\Hom_{\mathrm{cont}}(R,\Fp),$ and taking \(G\)-invariants is the same as killing conjugation by \(F\).  Therefore, we have
\[
        H^1(R,\Fp)^G
        =
        \Hom_{\mathrm{cont}}\bigl(R/R^p[R,F],\Fp\bigr).
\]
\end{proof}
To define the Zassenhaus filtration on $H^2(G,\Fp)$, we first introduce the corresponding filtration on the relation module. We then use the above duality to induce a Zassenhaus filtration on $H^2(G,\Fp)$.

\begin{definition}
     The Zassenhaus filtration on \(F\) induces the filtration on $R/R^p[R,F]$ as follows
\[
        \operatorname{Fil}_n(R/R^p[R,F]):=
        \im\left(R\cap D_nF\longrightarrow R/R^p[R,F] \right), \text{ for }{n\geq2 }.
\]
The graded relation modules are defined by  $\operatorname{gr}_n(R/R^p[R,F])=\frac{  \operatorname{Fil}_n(R/R^p[R,F])}{  \operatorname{Fil}_{n+1}(R/R^p[R,F])}\text{ for } {n\geq2 }$.
\end{definition}
\begin{remark}
    According to \cref{Zassenhau filtration} and \cite[Section 2.3]{GSduke}, $D_2F=F^p[F,F]$ and $R\subseteq D_2F$. Thus we have
\[
        R/R^p[R,F]= \operatorname{Fil}_2(R/R^p[R,F])\supseteq  \operatorname{Fil}_3(R/R^p[R,F])\supseteq \operatorname{Fil}_4(R/R^p[R,F])\supseteq\cdots.
\]For \(n\ge2\), let
\[
        s_n(G):=
        \dim_{\Fp}\operatorname{gr}_n(R/R^p[R,F]).
\]
By \cref{lem:relation mod dual}, we obtain $\sum_{n\ge2}s_n(G)
        =\dim_{\Fp}R/R^p[R,F]= \dim_{\Fp}H^2(G,\Fp)=r_p(G)$.
\end{remark}
\begin{definition}

\label{def:filtrationH2}
    Fix the isomorphism $  \phi_G: H^2(G,\Fp)\simeq (R/R^p[R,F])^\vee$. For $N\geq 1$, define
\begin{align*}
        \operatorname{Fil}_NH^2(G,\Fp):&=\phi_G^{-1}(( \operatorname{Fil}_{N+1}(R/R^p[R,F]))^\perp)\\
        &=
        \{\alpha\in H^2(G,\Fp) :\phi_G(\alpha)( \operatorname{Fil}_{N+1}(R/R^p[R,F]))=0\}.
\end{align*}
Define $\operatorname{gr}_NH^2(G,\Fp)
:=
\frac{\operatorname{Fil}_NH^2(G,\Fp)}{\operatorname{Fil}_{N-1}H^2(G,\Fp)}\text{ for }{N\geq 2}$.
\end{definition}

The following proposition describes the precise relationship between the Zassenhaus filtration on $H^2(G,\Fp) $ and the Zassenhaus filtration on the relation module $(R/R^p[R,F])^\vee$. In particular, it shows that the above duality is compatible with these filtration.

\begin{proposition}
\label{prop:fildual}
   For every \(N\ge2\), one has
\[
\operatorname{gr}_NH^2(G,\Fp)
\simeq
(\operatorname{gr}_N(R/R^p[R,F]))^\vee.
\]
Moreover, we have the following dimension formula
\[
        \dim_{\Fp}\operatorname{Fil}_NH^2(G,\Fp)
        =
        \sum_{n=2}^{N}s_n(G).
\]
\end{proposition}

\begin{proof}
For every \(N\geq 2,\) $\phi_G\bigl(\operatorname{Fil}_NH^2(G,\Fp)\bigr)
=
(\operatorname{Fil}_{N+1}(R/R^p[R,F]))^\perp.$ Therefore, \(\phi_G\) induces an isomorphism
\[
\operatorname{gr}_N H^2(G,\Fp)
=
\frac{\operatorname{Fil}_NH^2(G,\Fp)}
{\operatorname{Fil}_{N-1}H^2(G,\Fp)}
\xrightarrow{\sim}
\frac{(\operatorname{Fil}_{N+1}(R/R^p[R,F]))^\perp}
{(\operatorname{Fil}_{N}(R/R^p[R,F]))^\perp}
.
\]
Consider the restriction map
\[
(\operatorname{Fil}_{N+1}(R/R^p[R,F]))^\perp
\longrightarrow
\left(
\frac{\operatorname{Fil}_N(R/R^p[R,F])}
{\operatorname{Fil}_{N+1}(R/R^p[R,F])}
\right)^\vee,
\]
given by
\[
\lambda\longmapsto
\left(
x+\operatorname{Fil}_{N+1}(R/R^p[R,F])
\longmapsto
\lambda(x)
\right).
\]
This map is well defined and surjective. Its kernel consists precisely of
those $\lambda\in (R/R^p[R,F])^\vee$ which vanish on \(\operatorname{Fil}_N(R/R^p[R,F])\), that is $(\operatorname{Fil}_N(R/R^p[R,F]))^\perp.$ Therefore,
\[
\operatorname{gr}_N H^2(G,\Fp)
\simeq
\frac{(\operatorname{Fil}_{N+1}(R/R^p[R,F]))^\perp}
{(\operatorname{Fil}_{N}(R/R^p[R,F]))^\perp}
\simeq
\left(
\frac{\operatorname{Fil}_N(R/R^p[R,F])}
{\operatorname{Fil}_{N+1}(R/R^p[R,F])}
\right)^\vee.
\]
Since $\operatorname{Fil}_1H^2(G,\Fp)
=
\phi_G^{-1}\bigl((\operatorname{Fil}_2(R/R^p[R,F])^\perp\bigr)
=
\phi_G^{-1}(R/R^p[R,F]^\perp)
=0,$
we have
\begin{align*}
\dim_{\Fp}\operatorname{Fil}_NH^2(G,\Fp)
=
\sum_{n=2}^{N}
\dim_{\Fp}
\frac{\operatorname{Fil}_nH^2(G,\Fp)}
{\operatorname{Fil}_{n-1}H^2(G,\Fp)}=\sum_{n=2}^{N}
\dim_{\Fp}
\frac{\operatorname{Fil}_n(R/R^p[R,F])}
{\operatorname{Fil}_{n+1}(R/R^p[R,F])}
=\sum_{n=2}^{N}s_n(G).
\end{align*}
\end{proof}

To provide an effective method for computing the Golod--Shafarevich polynomials introduced in \cref{def:GS-grp}, we prove the following proposition.

\begin{proposition}\label{Defining relation}
There exists a minimal system of defining relations
\[\mathcal R=\{r_{n,j}:n\geq 2,\ 1\leq j\leq s_n(G)\}
        \subset R\]
adapted to the Zassenhaus filtration such that  $ r_{n,j}\in D_nF\setminus D_{n+1}F.$ In particular, with respect to this minimal system of relations, the corresponding Golod--Shafarevich polynomial is given by
\[
        P_G(\tau)=1-d_p(G)\tau+\sum_{n\geq 2}s_n(G)\tau^n.
\]
\end{proposition}

\begin{proof}
Since \(R/R^p[R,F]\) is finite-dimensional over \(\mathbb F_p\), there exists a positive integer \(N_0\) such that $R\cap D_{N_0}F\subseteq R^p[R,F]$ by \cite[Section 7.6]{Kochbook}. Thus, we have $ \operatorname{Fil}_{N_0}(R/R^p[R,F])=0.$ For each \(n\geq 2\), if $s_n(G)\neq 0$, then we can choose elements  $e_{n,1},\ldots,e_{n,s_n(G)}\in \operatorname{Fil}_n(R/R^p[R,F])$ whose images form a basis of $\operatorname{Fil}_n(R/R^p[R,F])/ \operatorname{Fil}_{n+1}(R/R^p[R,F]).$
Otherwise, no choice is needed.

We claim that all the elements \(e_{n,j}\), for all \(n\) and \(j\), form a basis of \(R/R^p[R,F]\). Suppose that there exist $a_{n,j}\in \mathbb F_p$ such that
\[\sum_{n,j} a_{n,j}e_{n,j}=0.\]
Assume that \(n_0\)  is the smallest integer such that \(a_{n_0,j}\neq 0\) for some $1\le j\le s_n(G)$. Reducing the above equation modulo \(\operatorname{Fil}_{n_0+1}(R/R^p[R,F])\), all terms with \(n>n_0\) vanish and
\[
        \sum_j a_{n_0,j}\overline e_{n_0,j}=0
\]
in $ \operatorname{Fil}_{n_0}(R/R^p[R,F])/\operatorname{Fil}_{n_0+1}(R/R^p[R,F]).$ Thus $a_{n_0,j}=0$ for all \(j\) which is a contradiction. Hence, \(e_{n,j}\) are linearly independent over $\mathbb F_p$. For any \(x\in R/R^p[R,F]=\operatorname{Fil}_2(R/R^p[R,F])\), the element \(x\) modulo \(\operatorname{Fil}_3(R/R^p[R,F])\) is a linear combination of the images of \(e_{2,j}\) with $1\le j\le s_2(G)$. Subtracting the corresponding linear combination of the \(e_{2,j}\), we get an element of \(\operatorname{Fil}_3(R/R^p[R,F])\). Repeating the same argument modulo \(\operatorname{Fil}_4(R/R^p[R,F])\), \(\operatorname{Fil}_5(R/R^p[R,F])\) and so on. Since the filtration is finite, we can express \(x\) as a linear combination of the \(e_{n,j}\). Thus, \(e_{n,j}\) form a basis of \(R/R^p[R,F]\).

Now, by the definition of \(\operatorname{Fil}_n(R/R^p[R,F])\), for each \(e_{n,j}\) we may choose an element $  r_{n,j}\in R\cap D_nF$ whose image in \(R/R^p[R,F]\) is \(e_{n,j}\). We claim that  $r_{n,j}\notin D_{n+1}F.$ If \(r_{n,j}\in D_{n+1}F\), then the image of \(r_{n,j}\) in \(R/R^p[R,F]\) would lie in $ \operatorname{Fil}_{n+1}(R/R^p[R,F]).$ Hence, \(e_{n,j}\in \operatorname{Fil}_{n+1}(R/R^p[R,F])\) whose image in $\operatorname{Fil}_n(R/R^p[R,F])/\operatorname{Fil}_{n+1}(R/R^p[R,F])$ is zero, a contradiction. Therefore $r_{n,j}\in D_nF\setminus D_{n+1}F,$ i.e., \(r_{n,j}\) has Zassenhaus degree \(n\).

It remains to show that the elements \(r_{n,j}\) form a minimal system of defining relations. Let \(N\) be the closed normal subgroup of \(F\) generated by all \(r_{n,j}\). Since \(r_{n,j}\in R\),  we have $ N\subseteq R.$ Moreover, since the images of the \(r_{n,j}\) form a basis of $    R/{R^p[R,F]},$  the following map \[H^1(R,\Fp)^G \longrightarrow H^1(N,\Fp)\]
is injective by the argument used in the proof of \cref{lem:relation mod dual}.
This implies \(N=R\) by \cite[Theorem 6.8]{Kochbook}. Hence, the elements \(r_{n,j}\)  form a minimal system of defining relations. For this adapted minimal system of relations, exactly \(s_n(G)\) relations have Zassenhaus degree \(n\). Hence, the Golod--Shafarevich polynomial associated with this adapted presentation is $ P_G(\tau)=1-d(G)\tau+\sum_{n\geq 2}s_n(G)\tau^n.$ This proves the proposition.
\end{proof}

\begin{remark}
   By \cref{prop:fildual}, one may choose a basis of \(H^2(G,\Fp)\) compatible with the induced filtration. More precisely, for each \(n\geq 2\), the images of a suitable subset of this basis form a basis of $\operatorname{gr}_n H^2(G,\Fp),$ and there is a natural isomorphism $\operatorname{gr}_n H^2(G,\Fp)
\cong
\operatorname{gr}_n\bigl(R/R^p[R,F]\bigr)^\vee.$ Moreover, by \cref{Defining relation}, after choosing an adapted minimal presentation of \(G\), a basis of $\operatorname{gr}_n\bigl(R/R^p[R,F]\bigr)$ corresponds to defining relations of Zassenhaus degree \(n\). Thus, the basis of \(H^2(G,\Fp)\) can be decomposed according to the Zassenhaus degrees of the corresponding defining relations.
\end{remark}

The following proposition will be useful in our estimates.

\begin{proposition}\label{prop:abelian-standard-relations}
Let $A=\bigoplus_{i=1}^d C_{p^{\alpha_i}}$, where $\alpha_i\ge1$. In the standard presentation
\[
A=\langle x_1,\ldots,x_d
\mid [x_i,x_j]\ (i<j),\
x_i^{p^{\alpha_i}}\ (1\le i\le d)\rangle,
\]
the defining relations form an adapted minimal system for the Zassenhaus filtration, with degrees $2$ and $p^{\alpha_i}$,
respectively. Consequently, for every $N\ge2$,
\[
\sum_{n=2}^N s_n(A)
=\binom d2+\#\{i:p^{\alpha_i}\le N\}.
\]
\end{proposition}

\begin{proof}
 Let \(F\) be the free pro-\(p\) group on \(x_1,\ldots,x_d\). The relation subgroup \(R\) is the closed normal subgroup of \(F\) generated by
\[
[x_i,x_j]\quad(i<j),
\qquad
x_i^{p^{\alpha_i}}\quad(1\le i\le d).
\]
Under the Magnus embedding \(\iota(x_i)=1+u_i\),
\[
\iota([x_i,x_j])-1
=
u_i u_j-u_j u_i+\text{terms of degree at least }3,
\]
and \(
\iota\!\left(x_i^{p^{\alpha_i}}\right)-1
=
u_i^{p^{\alpha_i}}.
\) Hence \(
D([x_i,x_j])=2\text{ and }
D\left(x_i^{p^{\alpha_i}}\right)=p^{\alpha_i}.
\) We claim that, for each \(n\ge2\), the standard relations of Zassenhaus degree \(n\) have linearly independent images in \(
\operatorname{gr}_n(R/R^p[R,F]).
\) For \(n=2\),  suppose that the images of the degree-two standard relations are linearly dependent in \(
\operatorname{gr}_2(R/R^p[R,F]).
\) Then there exist coefficients \(c_{ij}\in\Fp\) for \(i<j\) and, when $p=2,$ \(b_i\in\Fp\) for \(\alpha_i=1\), not all zero, such that
\[
\sum_{i<j}
c_{ij}\,\overline{[x_i,x_j]}
+
\mathbf 1_{\{p=2\}}
\sum_{\alpha_i=1}
b_i\,\overline{x_i^2}
=0
\]
in \(\operatorname{gr}_2(R/R^p[R,F])\), where the bars denote the corresponding classes in the second graded piece. Equivalently,
\[
\prod_{i<j}[x_i,x_j]^{c_{ij}}
\prod_{\alpha_i=1}(x_i^2)^{b_i}
\in
(R\cap D_3F)R^p[R,F],
\]
where the second product occurs only when \(p=2\). Since \(R\subseteq D_2F\), the Zassenhaus identities
\[
(D_mF)^p\subseteq D_{mp}F,
\qquad
[D_mF,D_\ell F]\subseteq D_{m+\ell}F
\]
give \(
R^p[R,F]\subseteq D_3F.
\)  Thus \(\prod_{i<j}[x_i,x_j]^{c_{ij}}
\prod_{\alpha_i=1}(x_i^2)^{b_i}\in D_3F\). Hence the homogeneous degree-two part of its Magnus expansion must vanish. On the other hand, by the above Magnus expansions, this
degree-two homogeneous part is precisely
\[
\sum_{i<j}c_{ij}(u_i u_j-u_j u_i)
+
\mathbf 1_{\{p=2\}}
\sum_{\alpha_i=1}b_i u_i^2,
\]
where \(\mathbf 1_{\{p=2\}}\) is \(1\) if \(p=2\) holds and \(0\) otherwise. Hence it is equal to \(0\).  Since the degree-two component of \(\Fp\langle u_1,\ldots,u_d\rangle\) has basis \(
\{\,u_i u_j:1\le i,j\le d\,\},
\) comparison of the coefficients of \(u_i u_j\) for \(i<j\) gives \(c_{ij}=0\), and, when \(p=2\), comparison of the coefficients of \(u_i^2\) gives \(b_i=0\). Hence all coefficients vanish. Thus the images of the degree-two standard relations are linearly independent in \(
\operatorname{gr}_2(R/R^p[R,F]).
\)

Now let \(n>2\). There is nothing to prove unless \(n=p^a\) for some
\(a\ge1\).  Suppose that the images of the degree $n$ standard relations are
linearly dependent in \(
\operatorname{gr}_n(R/R^p[R,F]).
\) Then  there exist coefficients \(
c_{i}\in\Fp \) , not all zero, such that 
\[
\prod_{\alpha_i=a}(x_i^{p^a})^{c_i}
\in
(R\cap D_{p^a+1}F)R^p[R,F].
\]
Let \(
\pi:F\to F^{\mathrm{ab}}\simeq \mathbb Z_p^d
\) be the abelianization map, and write \(
v_i=\pi(x_i)
\text{ for } 1\le i\le d.
\) Then
\[
\pi\!\left(
\prod_{\alpha_i=a}(x_i^{p^a})^{c_i}
\right)
=
\sum_{\alpha_i=a}c_i\,\pi(x_i^{p^a})
=
\sum_{\alpha_i=a}c_i p^a v_i.
\]
On the other hand, \(
\pi(R)
=
\bigoplus_{j=1}^{d} p^{\alpha_j}\mathbb Z_p v_j \text{ and }
\pi(D_{p^a+1}F)
\subseteq
D_{p^a+1}(\mathbb Z_p^d)
=
p^{a+1}\mathbb Z_p^d
\)  by \cref{Zassenhaus map} (i) and \cref{Zassenhau filtration}. Then
\[
\pi\!\left((R\cap D_{p^a+1}F)R^p[R,F]\right)
\subseteq
\pi(D_{p^a+1}F)+p\,\pi(R).
\]
Hence, for every \(i\) with \(\alpha_i=a\), comparison of the \(v_i\)-coordinate gives \(
c_i p^a\in p^{a+1}\mathbb Z_p.
\) Therefore \(c_i=0\) in \(\Fp\). Thus the degree-\(n\) standard relations are linearly independent in \(
\operatorname{gr}_n(R/R^p[R,F]).
\) Hence their images span \(R/R^p[R,F]\). Applying the same minimal-degree argument as in the proof of \cref{Defining relation}, their classes are linearly independent in \(R/R^p[R,F]\). Hence they form a basis of  the relation module adapted to the Zassenhaus filtration. Thus
\[
s_n(A)
=
\binom{d}{2}\mathbf 1_{\{n=2\}} +
\#\{\,i:p^{\alpha_i}=n\,\},
\]
where \(\mathbf 1_{\{n=2\}}\) is \(1\) if \(n=2\) holds and \(0\) otherwise. Consequently,
\[
\sum_{n=2}^{N}s_n(A)
=
\binom{d}{2}
+
\#\{\,i:p^{\alpha_i}\le N\,\}.
\]
\end{proof}

\subsection{Zassenhaus filtration on restriction map}
In this subsection, we assume that the finitely generated pro-$p$ groups $H$ and $G$ admit the minimal pro-$p$ presentation of the form,
\[
1\longrightarrow R_H\longrightarrow F_H\xrightarrow{\pi_H}H\longrightarrow 1
\]
and
\[
1\longrightarrow R_G\longrightarrow F_G\xrightarrow{\pi_G}G\longrightarrow 1.
\]
The goal of this section is to prove that the restriction map
\[
\operatorname{res}_H^G:H^2(G,\mathbb F_p)\longrightarrow H^2(H,\mathbb F_p)
\]
is filtered with respect to the filtration induced by the Zassenhaus filtration on the relation modules.

\begin{proposition}
\label{Zassenhaus map}
Let $\iota:H\longrightarrow G$ be a continuous homomorphism of pro-\(p\) groups.
\begin{enumerate}
\item For every \(n\ge1\), we have $\iota(D_nH)\subseteq D_nG.$
\item There exists a continuous homomorphism $\widetilde\iota:F_H\longrightarrow F_G$
such that \[\pi_G\circ\widetilde\iota=\iota\circ\pi_H.\]
Moreover, we have $\widetilde\iota(R_H)\subseteq R_G.$
\item The continuous homomorphism $\widetilde\iota:F_H\longrightarrow F_G$ induces a well-defined continuous homomorphism
\[\overline\iota:R_H/R_H^p[R_H,F_H]\longrightarrow R_G/R_G^p[R_G,F_G].\]
Moreover, for every \(n\ge1\), one has
\[\overline\iota(\operatorname{Fil}_n(R_H/R_H^p[R_H,F_H]))\subseteq \operatorname{Fil}_n(R_G/R_G^p[R_G,F_G]).\]
\item If $H$ is the closed subgroup of $G$, then the usual cohomological restriction map
\[
\operatorname{res}_H^G:H^2(G,\mathbb F_p)\to H^2(H,\mathbb F_p)
\]
satisfied that $\phi_H(\operatorname{res}_H^G\alpha)(m)=\phi_G(\alpha)(\overline\iota(m))$ for $m\in R_H/R_H^p[R_H,F_H]$ and $\alpha\in H^2(G,\mathbb F_p)$. Moreover,  for every \(N\ge1\), one has
\[
\operatorname{res}_H^G
\left(
\operatorname{Fil}_N H^2(G,\mathbb F_p)
\right)
\subseteq
\operatorname{Fil}_N H^2(H,\mathbb F_p).
\]
\end{enumerate}
\end{proposition}

\begin{proof}
\begin{enumerate}
    \item For \(i=1\), this is clear since $\gamma_1H=H,\gamma_1G=G$ and $\iota(H)\subseteq G.$ Assume inductively that $\iota(\gamma_iH)\subseteq \gamma_iG.$ Since $\gamma_{i+1}H=[\gamma_iH,H]$, we have
\[\iota(\gamma_{i+1}H)=\iota([\gamma_iH,H])\subseteq [\iota(\gamma_iH),\iota(H)]\subseteq [\gamma_iG,G]
=\gamma_{i+1}G.\]
Hence, we have $\iota(\gamma_iH)\subseteq\gamma_iG$ for all \(i\ge1\). It follows that $\iota\bigl((\gamma_iH)^{p^j}\bigr)
\subseteq (\gamma_iG)^{p^j}$ for every \(j\ge0\). By \cref{Zassenhau filtration}, we obtain
\[\iota(D_nH)=\iota\left(
\prod_{ip^j\ge n}(\gamma_iH)^{p^j}
\right)
\subseteq
\prod_{ip^j\ge n}(\gamma_iG)^{p^j}=D_nG.\]

\item Since the free pro-p group is the projective object of pro-p group category \cite[Theorem 4.8]{Kochbook}. Thus, the first claim is true. For any $r\in R_H$, i.e., $\pi_H(r)=1,$  we get
\[\pi_G(\widetilde\iota(r))
=(\pi_G\circ\widetilde\iota)(r)
=(\iota\circ\pi_H)(r)=\iota(\pi_H(r))=\iota(1)=1\]
by using the commutative diagram.
Thus, we have $\widetilde\iota(r)\in\ker(\pi_G)=R_G,$ i.e., $\widetilde\iota(R_H)\subseteq R_G.$

\item One can verify  $\widetilde\iota(R_H^p)\subseteq R_G^{p}$ and $\widetilde\iota([R_H,F_H])\subseteq [R_G,F_G].$ Thus, it can induce  a well-defined continuous homomorphism
\[\overline\iota:R_H/R_H^p[R_H,F_H]\longrightarrow R_G/R_G^p[R_G,F_G].\]
Moreover, by the definition of $ \operatorname{Fil}_n(R_H/R_H^p[R_H,F_H])$ and $\widetilde\iota(D_nF_H)\subseteq D_n F_G$ from item (i), one can show that \[
\overline\iota(\operatorname{Fil}_n(R_H/R_H^p[R_H,F_H]))\subseteq \operatorname{Fil}_n(R_G/R_G^p[R_G,F_G]).\]

\item  Let $\alpha\in H^2(G,\F_p)$ and $\varphi=\phi_G(\alpha)\in  (R_G/R_G^p[R_G,F_G])^\vee.$ We have the following map $$\overline\iota^\vee: (R_G/R_G^p[R_G,F_G])^\vee \longrightarrow (R_H/R_H^p[R_H,F_H])^\vee $$
by $\overline\iota^\vee(\varphi)(\bar r)=\varphi(\overline\iota(\bar r))$ for $\bar r\in R_H/R_H^p[R_H,F_H]$. By the naturality of the Hochschild--Serre five-term exact sequence applied to the morphism of presentations induced by $\widetilde\iota$, the transgression maps satisfy
\[
\res_H^G\circ\tg_G
=
\tg_H\circ\overline\iota^\vee,
\]
where \(
\tg_G:
(R_G/R_G^p[R_G,F_G])^\vee
\xrightarrow{\sim} H^2(G,\F_p)
\) is the transgression isomorphism arising from \cref{lem:relation mod dual}, and $\tg_H$ is defined analogously. Applying $\phi_H=\tg_H^{-1}$ and $\phi_G=\tg_G^{-1}$  gives
\[
\phi_G(\alpha)\circ\overline\iota=\overline\iota^\vee\varphi=\phi_H(\res_H^G(\tg_G(\varphi)))=\phi_H(\res_H^G(\alpha)).
\]
 For any $\alpha\in \operatorname{Fil}_N H^2(G,\mathbb F_p)$ and $m\in \operatorname{Fil}_{N+1}(R_H/R_H^p[R_H,F_H])$, we have
$$\phi_H(\operatorname{res}_H^G\alpha)(m)=\phi_G(\alpha)(\overline\iota(m))=0.$$
It follows that  $ \operatorname{res}_H^G\alpha\in  \operatorname{Fil}_N H^2(H,\mathbb F_p)$. Therefore, we have
\[
\operatorname{res}_H^G
\left(
\operatorname{Fil}_N H^2(G,\mathbb F_p)
\right)
\subseteq
\operatorname{Fil}_N H^2(H,\mathbb F_p).
\]
 
\end{enumerate}
\end{proof}

\begin{remark}
Taking $H=G$ and $\iota=\mathrm{id}_G$ in \cref{Zassenhaus map}, and applying the result in both directions to two minimal pro-$p$ presentations, shows that the filtration on $H^2(G,\F_p)$, and hence $s_n(G)$, is independent of the chosen presentation.
\end{remark}

\section{Lower bound for the generator rank}\label{sec:3}
In \cite[Section 3]{Iharacompositio}, Duursma and Mak summarized the result on the lower bound for the generator rank of the Galois group \(G_{S,\nu}\) of the maximal pro-\(p\) extension \(K_{S,\nu}\) of \(K\) which is unramified outside \(S\) and has ramification depth at most \(\nu_{\mathfrak p}\) at \(\mathfrak p\in S\) from \cite{HMbounded, Kuhnt}.
However, we calculate the generator rank of  the Galois group \(G_{S,\nu}^{T}\) of the maximal pro-\(p\) extension \(K_{S,\nu}^{T}\) of \(K\) which is unramified outside \(S\), has ramification depth at most \(\nu_{\mathfrak p}\) at \(\mathfrak p\in S\), and in which every prime of  \(T\) splits completely in this section.

\subsection{Ramification of bounded depth}
Let $K/\F_q$ be a global function field of characteristic $p$.  Let $S$ be a set of places of $K$ and let
\[ \nu:S\longrightarrow [0,\infty]\]
be a depth function, extended by $\nu_\gp=0$ for any place $p$ outside $S$.

\begin{definition}
Let $L/K$ be a Galois extension.  We say that $L/K$ has ramification of depth at most $n$ at $\gp$ if the upper ramification group $G^n_\gP$ is trivial for every place $\gP$ of $L$ lying above $\gp$.  Moreover, $L/K$ has ramification bounded by $\nu$ if the ramification of depth at each $\gp\in S$ is at most $\nu_\gp$.
\end{definition}

Let $K_{S,\nu}$ be the maximal $p$-extension of $K$ unramified outside $S$ and with ramification depth bounded by $\nu$, and set $G_{S,\nu}=\Gal(K_{S,\nu}/K).$ Moreover, if $T$ is a set of places of $K$ and let $K_{S,\nu}^T$ be the fixed field corresponding to the closed normal subgroup of $G_{S,\nu}$ generated by all decomposition group of $G_{S,\nu}$ at the places of $T$, then $K_{S,\nu}^T$ is the maximal $p$-extension of $K$ unramified outside $S$ with ramification depth bounded by $\nu$ and all the places in $T$ split completely.

If $L/K$ is a finite subextension and $\gP$ lies above $\gp$, then the induced depth on $L$ is obtained from the Herbrand transition
\[\nu_{L,\gP}=\psi_{\gP/\gp}(\nu_\gp),\]
where $\psi$ is given by $G_{\mathfrak{P}}^v=(G_{\mathfrak{P}})_{\psi_{\gP/\gp}(v)}$ relating the upper and lower ramification groups.

The extension $K_{S,\nu}/K$ can be obtained by repeatedly taking maximal abelian layers subject to the same ramification-depth condition by \cite[Proposition 3.1]{Iharacompositio}.  Moreover $K_{S,\nu}^T/K$  has the similar property, since the decomposition group can be viewed as the local Galois group and we can use the similar proof  from \cite[Theorem 3.5]{HMbounded}. This reduces the calculation of the generator rank to class field theory.

\subsection{Idelic calculation of the generator rank}
For a place $\gp$ of $K$, let $U_\gp$ be the unit group of local field $K_\gp$. For $n\geq 1$, define the $n$-th unit group of $K_\gp$ by
$U^{(n)}_\gp=\{x\in K_\gp: v_\gp(x-1)\ge n\}.$ Let \(S\) and \(T\) be two disjoint finite sets of places of \(K\).
Now let us fix the following notations
\[
\Delta
=
\left\{
x\in K^*
\;\middle|\;
\operatorname{div}(x)\in p\,\operatorname{Div}(K)
\right\}/K^{*p},
\]
\[
\Delta_{\emptyset}^{T}=\left\{
x\in K^*\ \middle|\
x\in K_{\mathfrak q}^{*p}U_{\mathfrak q}^{}
\text{ for all } \mathfrak q\notin T
\right\}/K^{*p}.
\]
If $T=\emptyset$, we write $\Delta_{\emptyset}$ for $\Delta_{\emptyset}^{T}.$  Then, we have $\Delta_{\emptyset}=\Delta$ by definition. Moreover, we define
\[
\Delta_{S}^{T}=\left\{x\in K^*\ \middle|\
\begin{array}{l}
x\in K_{\mathfrak p}^{*p}
\quad \text{for all } \mathfrak p\in S,
x\in K_{\mathfrak q}^{*p}U_{\mathfrak q}^{}
\quad \text{for all } \mathfrak q\notin  S\cup T
\end{array}
\right\}
/K^{*p},\]
\[
\Delta_{S,\nu}^{T}=\left\{x\in K^*\ \middle|\
\begin{array}{l}
x\in K_{\mathfrak p}^{*p}
U_{\mathfrak p}^{(\nu_{\mathfrak p})}
\quad \text{for all } \mathfrak p\in S,
x\in K_{\mathfrak q}^{*p}U_{\mathfrak q}^{}
\quad \text{for all } \mathfrak q\notin S\cup T
\end{array}
\right\}
/K^{*p}.
\]

\begin{theorem}
\label{thm:generator rank}
Let \(K\) be a global function field over \(\mathbb F_q\), where \(q=p^e\). Let \(S\) and \(T\) be two disjoint finite sets of places of \(K\). For each \(\mathfrak p\in S\), let \(\nu_{\mathfrak p}\in \Z_{\geq1}\) and $\deg \gp=f_\gp$. Let  \(G_{S,\nu}^{T}\) be the Galois group of the maximal pro-\(p\) extension of \(K\) which is unramified outside \(S\), has ramification depth at most \(\nu_{\mathfrak p}\) at \(\mathfrak p\in S\), and in which every prime of  \(T\) splits completely. Then the generator rank of $G_{S,\nu}^{T}$ is given by
\[
d_p(G_{S,\nu}^{T})
=
1
-|T|
+d_p(\Delta_{S,\nu}^{T})
+ \sum_{\mathfrak p\in S}
e f_\gp\bigl(\nu_\gp-1-\floor{(\nu_\gp-1)/p}\bigr).
\]
\end{theorem}

\begin{proof}
Let $J_K$ be the group of ideles of $K$. Let
\[
U_{\emptyset}^{T}
=
\prod_{\mathfrak t\in T}K_{\mathfrak t}^{*}
\prod_{\mathfrak q\notin T}U_{\mathfrak q}^{}\quad \text{ and  } \quad
U_{S,\nu}^{T}
=
\prod_{\mathfrak p\in S}U_{\mathfrak p}^{(\nu_{\mathfrak p})}
\prod_{\mathfrak t\in T}K_{\mathfrak t}^{*}
\prod_{\mathfrak q\notin S\cup T}U_{\mathfrak q}^{}.
\]
The class field theory \cite[Proposition 2.4.4 ]{HXrationalbook} gives the following exact sequence
\[
1
\longrightarrow
\Delta_{S,\nu}^{T}
\longrightarrow
\Delta_{\emptyset}^T
\longrightarrow
\frac{U_{\emptyset}^T}{(U_{\emptyset}^T)^pU_{S,\nu}^{T}}
\longrightarrow
\frac{J_K}
{K^{*}U_{S,\nu}^{T}J_K^p}
\longrightarrow
\frac{J_K}
{K^{*}U_{\emptyset}^T J_K^p}
\longrightarrow
1.
\]
Hence, after taking \(p\)-ranks, we obtain
\[d_p\left(\frac{J_K}
{K^{*}U_{S,\nu}^{T}J_K^p}\right)
=d_p\left(\frac{J_K}
{K^{*}U_{\emptyset}^T J_K^p}\right)
+d_p\left(
\frac{U_{\emptyset}^T}{(U_{\emptyset}^T)^pU_{S,\nu}^{T}}
\right)
-d_p(\Delta_{\emptyset}^T)
+d_p(\Delta_{S,\nu}^{T}).\]
Since $U_{\mathfrak p}/U_{\mathfrak p}^{(1)}$ has order prime to \(p\), we have
\[
d_p\left(
\frac{U_{\emptyset}^T}{(U_{\emptyset}^T)^pU_{S,\nu}^{T}}
\right)
=\sum_{\mathfrak p\in S}d_p
\frac{U_{\mathfrak p}}
{(U_{\mathfrak p})^p
U_{\mathfrak p}^{(\nu_{\mathfrak p})}}
=
\sum_{\mathfrak p\in S}
d_p\left(
\frac{U_{\mathfrak p}^{(1)}}
{U_{\mathfrak p}^{(\nu_{\mathfrak p})}}
\right).
\]
It remains to compute \(d_p(\frac{J_K}
{K^{*}U_{\emptyset}^T J_K^p})\). There is an exact sequence
\[
0
\longrightarrow
\Delta_{\emptyset}
\longrightarrow
\Delta_{\emptyset}^T
\longrightarrow
\bigoplus_{\mathfrak t\in T}\mathbb F_p\mathfrak t
\longrightarrow
\operatorname{Pic}_K/p\operatorname{Pic}_K
\longrightarrow
\frac{J_K}
{K^{*}U_{\emptyset}^T J_K^p}
\longrightarrow
0.
\]
All homomorphisms are obtained in a canonical and obvious manner, moreover you can find the above exact sequence in \cite[Proposition 2.3]{Mild} or \cite[Lemma 2.2]{maire2025ssplitphilbertclassfield}. Let $\operatorname{Cl}_K$ be the degree-zero divisor class group of $K$. Since $d_p(\F_q^*)=0$ and \cite[Theorem 2.4.1]{HXrationalbook}, we have $d_p(\Delta_{\emptyset})=d_p(\operatorname{Cl}_K)=d_p(\operatorname{Pic}_K)-1$ and
\[
d_p\left(\frac{J_K}
{K^{*}U_{\emptyset}^T J_K^p}\right)
=
1+d_p(\Delta_{\emptyset}^T)-|T|.
\]
Let $G_{S,\nu}^{T,abel}$ denote the maximal abelian quotient of \(G_{S,\nu}^{T}\). Substituting this into the previous formula yields
\[d_p(G^T_{S,\nu})=d_p(G_{S,\nu}^{T,abel})
=d_p\left(\frac{J_K}
{K^{*}U_{S,\nu}^{T}J_K^p}\right)=1+d_p(\Delta_{S,\nu}^{T})
+\sum_{\mathfrak p\in S}d_p\left(
\frac{U_{\mathfrak p}^{(1)}}
{U_{\mathfrak p}^{(\nu_{\mathfrak p})}}\right)-|T|.\]
By \cite[Lemma 4.2.5]{HXrationalbook}, we have
\[
d_p\left(U^{(1)}_\gp/U^{(\nu_\gp)}_\gp\right)
=
e f_\gp\left(\nu_\gp-1-\floor{\frac{\nu_\gp-1}{p}}\right).
\]
This completes the proof.
\end{proof}

\section{Zassenhaus Relation Defect Relative to the Generator Rank }
\label{sec:relation-detailed}

Throughout this section $L/K/k$ is a tower of Galois $p$-extensions of global function fields over a finite field.  The extension $K/k$ is finite and ramified only above a finite set $S$ of places of $k$, while $L/K$ is unramified.  If $\gp\in S$, we write $G_\gp$ for a local Galois group of $G=\Gal(L/k)$ at a chosen place of $L$ above $\gp$, and $I_\gp$ for its inertia subgroup.  The notations
\[
   d_p(G)=\dim_{\F_p}H^1(G,\Z/p\Z),
   \quad \text{ and } \quad
   r_p(G)=\dim_{\F_p}H^2(G,\Z/p\Z)
\]
stands for the generator rank and relation rank of finite $p$-groups or finitely generated pro-$p$ groups $G$ respectively in the usual continuous cohomology sense.
This section expands the cohomological estimate used in the class field tower argument. The aim is to bound the difference
\[
\sum_{n=2}^{N}s_n(G)-d_p(G),
\]
where $\sum_{n=2}^{N}s_n(G)$ is the $\F_p$-dimension of $\operatorname{Fil}_N H^2(G,\Z/p\Z)$ defined in Section \ref{sec:2}. For finite $G$, this sum equals $r_p(G)$ for all sufficiently large $N$.
We first introduce some result in the unramified cohomology which can seen in \cite[Section 4.3]{Iharacompositio}. Define
\begin{equation}\label{eq:unramified-cohomology}
   H^i_{\mathrm{nr}}(G_\gp,\Z/p\Z)
   =\operatorname{im}\left(
      H^i(G_\gp/I_\gp,\Z/p\Z)
      \xrightarrow{\infmap}
      H^i(G_\gp,\Z/p\Z)
   \right).
\end{equation}
This is the part of local cohomology that comes from the unramified quotient. Now assume that $L/k$ is a $p$-extension ramified only inside $S$.  Define
\[
   S_{\mathrm{nr}}
   =\left\{
      \gp\in S:
      d_p(I_\gp G'_\gp/G'_\gp)=d_p(G_\gp)-1
     \right\},
\]
where $G'_\gp$ is the commutator subgroup of $G_\gp$. Equivalently, if $L_\gp^{\mathrm{ab}}/k_\gp$ denotes the maximal abelian subextension of the local extension, then $\gp\in S_{\mathrm{nr}}$ exactly when
\[
   d_p(G_\gp)=d_p(G_\gp^{\mathrm{ab}})=d_p(I_\gp^{\mathrm{ab}})+1.
\]
For $\gp\notin S_{\mathrm{nr}}$, one has instead
\[d_p(G_\gp)=d_p(I_\gp^{\mathrm{ab}}).\]
We equip  \(H^2_{nr}(G_{\mathfrak{p}},\Z/p\Z)\)  with the induced filtration as follows:
\[
        \operatorname{Fil}_NH^2_{nr}(G_{\mathfrak{p}},\Z/p\Z):=H^2_{nr}(G_{\mathfrak{p}},\Z/p\Z)\cap \operatorname{Fil}_NH^2_{}(G_{\mathfrak{p}},\Z/p\Z)
\]
 for $N\geq 2$. Let \(u_n(G_{\mathfrak p})\) denote the dimension of the degree-\(n\) graded piece of the filtered group \(H^2_{\mathrm{nr}}(G_{\mathfrak p},\mathbb Z/p\mathbb Z)\), i.e., $u_2(G_{\mathfrak p}):=\dim_{\mathbb F_p}
\operatorname{Fil}_2 H^2_{\mathrm{nr}}(G_{\mathfrak p},\mathbb Z/p\mathbb Z)$ and
for \(n\geq 3\),
\[
u_n(G_{\mathfrak p})
:=
\dim_{\mathbb F_p}
\frac{\operatorname{Fil}_n H^2_{\mathrm{nr}}(G_{\mathfrak p},\mathbb Z/p\mathbb Z)}
{\operatorname{Fil}_{n-1} H^2_{\mathrm{nr}}(G_{\mathfrak p},\mathbb Z/p\mathbb Z)}.
\]

Choose a filtration-adapted basis of $H^2_{\mathrm{nr}}(G_{\mathfrak p},\Z/p\Z)$ with respect to the filtration induced from $H^2(G_{\mathfrak p},\Z/p\Z)$. This basis can be extended to a filtration-adapted basis of $H^2(G_{\mathfrak p},\Z/p\Z)$. By \cref{lem:relation mod dual,prop:fildual,Defining relation}, the corresponding dual basis can be lifted to an adapted minimal system of defining relations of $G_{\mathfrak p}$. The defining relations corresponding to the chosen basis of $H^2_{\mathrm{nr}}(G_{\mathfrak p},\Z/p\Z)$ are called the \emph{unramified defining relations}. These form a subset of an adapted minimal system of defining relations of $G_{\mathfrak p}$. Since the filtration is induced, we have $u_n(G_{\mathfrak p})\leq s_n(G_{\mathfrak p})$ for every $n\geq 2$. Moreover, we have the following proposition from \cref{prop:fildual} and \cite[Proposition 4.4]{Iharacompositio}.

\begin{proposition}
\label{prop:local-hnr}
For $\gp\in S$, we have$$\dim_{\F_p}\operatorname{Fil}_NH^2_{nr}(G_{\mathfrak{p}},\Z/p\Z)= \sum_{n=2}^{N}u_n(G_{\mathfrak{p}}).$$
Moreover, for all sufficiently large $N$, we have
\[
 \operatorname{Fil}_NH^2_{nr}(G_{\mathfrak{p}},\Z/p\Z)
 =
 H^2_{\mathrm{nr}}(G_\gp,\Z/p\Z)
   \simeq
   \begin{cases}
      \Z/p\Z, & \text{ if } \gp\in S_{\mathrm{nr}};\\
      0,      & \text{ if } \gp\notin S_{\mathrm{nr}}.
   \end{cases}
\]
\end{proposition}
Let $L/K/k$ be a tower of Galois $p$-extensions of global function fields of characteristic $p$. Assume that $K/k$ is finite and unramified outside a nonempty finite set $S$ of places of $k$, and that $L/K$ is unramified. Let $T_k$ be a nonempty finite set of places of $k$, disjoint from $S$, such that every place in $T_k$ splits completely in $L/k$. Let $T_K$ and $T_L$ denote the sets of all places of $K$ and $L$, respectively, lying above $T_k$. Assume in addition that $T_k$ contains a rational place.

Under the above notation, our aim is to establish an upper bound for $\sum_{n=2}^{N}s_n(G)-d_p(G)$. Our main result is the following theorem.

\begin{theorem}\label{thm:kuhnt}
Under the above notation and assumptions, for every $N\geq 2$, define $\delta_N$ by
\begin{align}
\delta_N := {}&
\sum_{n=2}^{N}s_n(G)-d_p(G)
-\sum_{\mathfrak p\in S\setminus S_{\mathrm{nr}}}
\left(
\sum_{n=2}^{N}s_n(G_{\mathfrak p})
-d_p(G_{\mathfrak p})
\right) \notag\\
&-\sum_{\mathfrak p\in S_{\mathrm{nr}}}
\left(
\sum_{n=2}^{N}
(s_n(G_{\mathfrak p})-u_n(G_{\mathfrak p}))
-d_p(G_{\mathfrak p})+1
\right) \notag\\
&+d_p(\Delta_S^{T_k})+1-|T_k|.
\label{eq:N0-initial}
\end{align}

If $\delta_N>0$, then there is a nontrivial unramified Galois $p$-extension $\widetilde L/L$ in which every place of $T_L$ splits completely and whose constant field is the same as that of $L$. More precisely, the construction gives an elementary abelian quotient of $\Gal(\widetilde L/L)$ of $p$-rank at least $\delta_N$.

Consequently, if $L/K$ is the maximal unramified pro-$p$ extension of $K$ in which every place of $T_K$ splits completely, then
\begin{align*}
\sum_{n=2}^{N}s_n(G)-d_p(G)
&\leq
\sum_{\mathfrak p\in S\setminus S_{\mathrm{nr}}}
\left(
\sum_{n=2}^{N}s_n(G_{\mathfrak p})
-d_p(G_{\mathfrak p})
\right)\\
&\quad+
\sum_{\mathfrak p\in S_{\mathrm{nr}}}
\left(
\sum_{n=2}^{N}
(s_n(G_{\mathfrak p})-u_n(G_{\mathfrak p}))
-d_p(G_{\mathfrak p})+1
\right)\\
&\quad-d_p(\Delta_S^{T_k})-1+|T_k|.
\end{align*}
\end{theorem}

Before we prove the following theorem, we need to introduce some techniques about the global embedding problems which can produce many central elementary abelian extensions above $L$.

\subsection{The embedding problem}

Let $G$ be a profinite group.  An embedding problem for $G$ consists of a surjective continuous homomorphism  $\theta:\mathfrak{G}\twoheadrightarrow G$ and an exact sequence of finite groups
\[
   0\longrightarrow A\longrightarrow E\stackrel{j}{\longrightarrow}G\longrightarrow 1,
\]
where $A$ is abelian.  A solution of the above embedding problem is a continuous homomorphism $\chi:\mathfrak{G}\to E$ such that $j\circ \chi=\theta$.  It is called proper if $\chi$ is surjective. The above exact sequence defines a class $\epsilon\in H^2(G,A).$ The obstruction to solving the embedding problem is the inflation  $\inf{}_\mathfrak{G}^{G}(\epsilon):H^2(G,A)\longrightarrow H^2(\mathfrak{G},A)$ induced by $  \theta:\mathfrak{G}\twoheadrightarrow G$. The embedding problem has a solution if and only if $ \inf{}_\mathfrak{G}^{G}(\epsilon)=1
   \quad\text{in } H^2(G,A)$  from \cite[Proposition 4.1]{Iharacompositio}.

In our case, let $L/k$ be a finite Galois $p$-extension with group $G=\Gal(L/k)$, let $S_k$ be a finite nonempty set of places of $k$ containing all places ramified in $L/k$ and let \(T_k\) be a finite nonempty set of places of \(k\), disjoint from \(S_k\), such that every place in \(T_k\) splits completely in \(L/k\). Let \(S_L\) and \(T_L\) denote the sets of all places of \(L\) lying above \(S_k\) and \(T_k\). Let $ \mathfrak{G}= G^{T_k}_{S_k,p}=\Gal(k^{T_k}_{S_k,p}/k)$,  where  $k^{T_k}_{S_k,p}$ is the maximal pro-$p$ extension of $k$ unramified outside  $S_k$ and in which all places of $T_k$ split completely. Now we denote the above embedding problem by$$(L/k,\epsilon,S_L,T_L).$$
If $S_k$ is nonempty, then  $H^2(\mathfrak{G},\Z/p\Z)=0 $ by \cite[Theorem 1.1]{izv}. Thus, $\inf{}_\mathfrak{G}^{G}(\epsilon)=1$ and the above embedding problem has a solution. If the corresponding exact sequence is nonsplit, then we will get a proper solution correspond to a proper field $M$ such that $\Gal(M/k)=E $, which is unramified outside $S_k$ and in which all places of $T_k$  split completely. The following proposition can be found in \cite[Proposition 4.2]{Iharacompositio}.

\begin{proposition}[Solvability of the global problem]
\label{prop:global-embedding-solvable}
Let $L/k$ be a Galois $p$-extension, and let $S_L,T_L$ be as above.  If
\[
   0\longrightarrow \Z/p\Z\longrightarrow E\longrightarrow \Gal(L/k)\longrightarrow 1
\]
is a nonsplit central extension (meaning that is a nonsplit exact sequence), then the corresponding global embedding problem $$(L/k,\epsilon,S_L,T_L)$$
has a proper solution field $M$ with the same constant field as $L$ such that  $\Gal(M/L)= \Z/p\Z $, unramified outside $S_L,$ in which all places of  $T_L$  split completely.
\end{proposition}
For each nonzero class $\epsilon\in H^2(G,\mathbb Z/p\mathbb Z)$, the preceding proposition produces a proper solution field $M$. The next two propositions explain why, if one wants to construct many central elementary abelian extensions above $L$, the classes $\epsilon$ should be chosen to be linearly independent. More details about this part can be found in \cite[Chapter 3 and Chapter 4]{Kuhnt}.

\begin{proposition}\cite[\S6.6]{Weibel}
\label{prop:central extension}
Let $G$ be a group and $A$ an abelian group equipped with the trivial $G$-action. Let $E(G,A)$ denote the set of equivalence classes of central extensions of  $G$ by $A$. Then there is a natural bijection
\[
E(G,A)
\;\xleftrightarrow{\;\sim\;}
H^2(G,A).
\]
\end{proposition}

\begin{proposition}[Independence of marked proper solutions]
\label{Marked solution}
Let \(L/k\), \(S_L\), \(T_L\) be defined as above, and put \(G=\Gal(L/k)\). Let $ \epsilon_1,\ldots,\epsilon_n\in H^2(G,\mathbb F_p)$ be linearly independent. For each \(i\), let \(M_i\) be a proper solution field of the marked embedding problem
\[
        (L/k,\epsilon_i,S_L,T_L).
\]
Put $M=M_1M_2\cdots M_n .$ Then we have  $ d_p(\Gal(M/L))=n $ and $d_p(\Gal(M/k))=d_p(G)$. Moreover, \(M/L\) is unramified outside \(S_L\), and every place of \(T_L\) splits completely in \(M/L\).
\end{proposition}

\begin{proof}The first assertion $ d_p(\Gal(M/L))=n$  follows from \cite[Proposition 4.3]{Iharacompositio} and the second assertion $d_p(\Gal(M/k))=d_p(G)$ follows from \cite[Theorem 3.18]{Kuhnt}. Each \(M_i/L\) is unramified outside \(S_L\), and every place of \(T_L\) splits completely in \(M_i/L\). These two properties are preserved under compositum. Hence, \(M/L\) is unramified outside \(S_L\) and every place of \(T_L\) splits completely in \(M/L\).
\end{proof}

\subsection{Class field theory of wildly ramified extensions}

In this subsection, we recall several ramification facts that will be used to control the truncation of high ramification given in \cite{Iharacompositio, Kuhnt}.

\begin{lemma}
\cite[Lemma 4.8]{Iharacompositio}\label{lem:base-change-depth} Let $L/k$ be a Galois extension, and let $k'/k$ be a Galois extension linearly disjoint from $L$.  If $k'/k$ has ramification depth at most $v$ at $\gp$, then $Lk'/L$ has ramification depth at most
\[
   \psi_{L/k,\gp}(v)
\]
at every place of $L$ above $\gp$, where $\psi$ is the inverse Herbrand function as defined in Section 3.1.
\end{lemma}

\begin{lemma}
\cite[Lemma 4.9]{Iharacompositio}\label{lem:jumps-prime-to-p} Let $E/F$ be an elementary abelian $p$-extension of global function fields of characteristic $p$.  Then every upper ramification jump of $E/F$ is prime to $p$.
\end{lemma}

\begin{lemma}
\cite[Lemma 4.10]{Iharacompositio}\label{lem:inertia-rank-bound} Let $E/F$ be a $p$-extension of global function fields of characteristic $p$, with exact constant field $\F_{p^e}$.  Let $\gp$ be a place of $F$ of degree $f_\gp$, and assume that the local extension is abelian and has ramification depth at most \(\nu_{\mathfrak p}\in \Z_{\geq1}\).  Then the inertia group $I_\gp$ satisfies
\[
   d_p(I_\gp)
   \le e f_\gp\bigl(\nu_\gp-1-\floor{(\nu_\gp-1)/p}\bigr).
\]
\end{lemma}

\subsection{The proof of
\texorpdfstring{\cref{thm:kuhnt}}
{Theorem \ref{thm:kuhnt}}}

Let  $L/K/k$ be a tower of Galois \(p\)-extensions of global function fields of characteristic \(p\). We say that a Galois extension $L/K$ is central over the Galois extension $K/k$, if $\Gal(L/K)$ is contained in the center of $\Gal(L/k)$. Let \(T_k\) be a finite nonempty set of places of \(k\). Assume all the places of $T_k$ split completely in $L/k$. Let $ G=\Gal(L/k)$. Assume that \(K/k\) is finite and unramified outside a finite set \(S\) of places of \(k\) and the ramification depth of \(L/k\) is bounded by $ \nu=(\nu_{\mathfrak p})_{\mathfrak p\in S}.$ Define the global unramified part of $H^2(G,\Z/p\Z)$ by
\[
   H^2_{\mathrm{nr}}(G,\Z/p\Z)
   =\{\epsilon\in H^2(G,\Z/p\Z):
       \res^G_{G_\gp}(\epsilon)
       \in H^2_{\mathrm{nr}}(G_\gp,\Z/p\Z)
       \text{ for every }\gp\in S\}.
\]
Let $V$ be the subspace of $H^2_{nr}(G,\Z/p\Z)$ generated by a basis  $ \epsilon_1,\ldots,\epsilon_n$. Thus,  each $ \epsilon_i$ gives a proper global embedding solution field $Y_i$  of \[
        (L/k,\epsilon_i,S_L,T_L)
\] by \cref{Marked solution}. Let $N=Y_1\cdots Y_n$. Now all the places in $T_k$ split completely in $N/k$ and the constant field of $N$ is the same as $L.$

\begin{proposition}
\label{prop:Truncation of high ramification}
Under the above notations.
\begin{enumerate}
    \item There exists a finite extension \(M/L\) such that  \(M/k\) is a Galois extension, \(M/L\) is central over \(L/k\) and the ramification depth of \(M/L\) is bounded by the lifted bound \(\nu_L\). Moreover, all the places of  $T_k$ split completely in $M/k$ and
\[d_p(\Gal(M/L))\ge d_p(\Gal(N/L))+d_p(\Delta_S^{T_k})-d_p(\Delta_{S,\nu}^{T_k})
        +d_p(G_{S,\nu}^{T_k})-d_p(G).\]
\item There exists a finite unramified extension \(\tilde{L}/L\), Galois over \(k\), such that \(\tilde{L}/L\) is central over \(\tilde{L}/k\). Moreover, all the places of  $T_k$ split completely in $\tilde{L}/k$ and
$$ d_p(\Gal(\tilde{L}/L))\ge   d_p(\Gal(N/L)) +d_p(\Delta_S^{T_k}) +1-|{T_k}|
        -d_p(G)+\sum_{\mathfrak q\in S} d_p(G_\mathfrak q)-|S_{nr}|.$$

\end{enumerate}
\end{proposition}

\begin{proof}
\begin{enumerate}

   \item   \medskip
\noindent\textbf{ Step 1. Add the elementary abelian extensions of $k$ at the highest ramification depth of $N/L$.}
   
Let  $ S_L=
        \{\mathfrak P\in \mathbb{P}_L:\mathfrak P\mid \mathfrak p
        \text{ for some }\mathfrak p\in S\}.$ For \(\mathfrak P\in S_L\), let $\nu_{L,\mathfrak P}$ be the lifted ramification bound.  Suppose that the ramification depth of \(N/L\) exceeds \(\nu_L\) at some place \(\mathfrak Q\in S_L\). Put $\mathfrak q=\mathfrak Q\cap k.$ Let $L_{\mathfrak Q}^{\nr}$ be the maximal unramified subextension of \(L_{\mathfrak Q}/k_{\mathfrak q}\). Since the elementary abelian extension \(N/L\) comes from unramified cohomology, and the local unramified cohomology group $H^2_{\mathrm{nr}}(G_\mathfrak q,\Z/p\Z)$ is the image of $H^2(G_\mathfrak q/I_\mathfrak q,\Z/p\Z)$ under the inflation map. Hence, by \Cref{prop:central extension} and \cite[Exercise~6.6.4]{Weibel}, we obtain the following commutative diagram:
\begin{equation}\label{Diagram:inflation}
\begin{tikzcd}[ampersand replacement=\&,cramped]
	1 \& {\mathrm{Gal}(N_{\mathfrak Q}/L_{\mathfrak Q})} \& {\mathrm{Gal}(N_{\mathfrak Q}/k_{\mathfrak q})} \& {G_{\mathfrak q}} \& 1 \\
	1 \& {(\mathbb Z/p\mathbb Z)^{r_{\mathfrak Q}}} \& {\mathrm{Gal}(N^0_{\mathfrak Q}/k_{\mathfrak q})} \& {G_{\mathfrak q}/I_{\mathfrak q}} \& {1,}
	\arrow[from=1-1, to=1-2]
	\arrow[from=1-2, to=1-3]
	\arrow["\sim"', from=1-2, to=2-2]
	\arrow[from=1-3, to=1-4]
	\arrow[from=1-3, to=2-3]
	\arrow[from=1-4, to=1-5]
	\arrow[two heads, from=1-4, to=2-4]
	\arrow[from=2-1, to=2-2]
	\arrow[from=2-2, to=2-3]
	\arrow[from=2-3, to=2-4]
	\arrow[from=2-4, to=2-5]
\end{tikzcd}
\end{equation}
where \(
r_{\mathfrak Q}
=\dim_{\mathbb F_p}
\Gal(N_{\mathfrak Q}/L_{\mathfrak Q}).
\) Since $G_{\mathfrak q}/I_{\mathfrak q}$ is cyclic and the lower extension is central, $ \Gal(N^0_{\mathfrak Q}/k_{\mathfrak q})$ is abelian.  Consequently, $N_{\mathfrak Q}/L_{\mathfrak Q}$ is obtained by base change from the elementary abelian extension $N^0_{\mathfrak Q}/L_{\mathfrak Q}^{\nr}$, as illustrated in \cref{fig:truncation-field-diagram}.  Let $ n_{\mathfrak q}$ be the ramification depth of $N^0_{\mathfrak Q}/L_{\mathfrak Q}^{\nr}.$ Then $n_{\mathfrak q}-1$ is the highest upper ramification jump. Since $L_{\mathfrak Q}^{\nr}/k_{\mathfrak q}$ is unramified, thus  \(n_{\mathfrak q}-1\) is also the highest upper ramification jump over \(N^0_{\mathfrak Q}/k_{\mathfrak q}\) \cref{lem:base-change-depth} . Let $\mathbf n=(n_{\mathfrak p})_{\mathfrak p\in S}$ be the depth vector, and define
\[
        \mathbf n^{-}=(n^-_{\mathfrak p})_{\mathfrak p\in S}\text{ where  }   n^-_{\mathfrak p}
        =
        \begin{cases}
        n_{\mathfrak p}, & \mathfrak p\neq \mathfrak q,\\
        n_{\mathfrak q}-1, & \mathfrak p=\mathfrak q.
        \end{cases}
\]
Let \(
V_{\mathbf n}
=H^1(G_{S,\mathbf n}^{T_k},\mathbb F_p),
V_{\mathbf n^-}
=H^1(G_{S,\mathbf n^-}^{T_k},\mathbb F_p).
\) Choose an $\mathbb F_p$-complement $W$ of $V_{\mathbf n^-}$ in $V_{\mathbf n}$, and let $E_{\mathbf n,\mathbf n^-}/k$ be the elementary abelian $p$-extension corresponding to $W$. Every nonzero character in $W$ has ramification depth exactly $n_{\mathfrak q}$ at $\mathfrak q$. Consequently,
\[
\begin{aligned}
d_p\Gal(E_{\mathbf n,\mathbf n^-}/k)
=\dim_{\mathbb F_p}W
=d_p(G_{S,\mathbf n}^{T_k})
-d_p(G_{S,\mathbf n^-}^{T_k}).
\end{aligned}
\]
By \cref{Marked solution}, $d_p\Gal(N/k)=d_p(G)$. Hence every elementary abelian subextension of $N/k$ is contained in $L$ and has ramification depth bounded by $\nu\leq\mathbf n^-$. The equality $W\cap V_{\mathbf n^-}=0$ gives \(
N\cap E_{\mathbf n,\mathbf n^-}=k.
\) Put $M_1=NE_{\mathbf n,\mathbf n^-}$. It follows that $N\cap E_{\mathbf n,\mathbf n^-}L=L$. Therefore,
\[
\begin{aligned}
d_p\Gal(M_1/L)
&=d_p\Gal(N/L)
+d_p(G_{S,\mathbf n}^{T_k})
-d_p(G_{S,\mathbf n^-}^{T_k}),
\end{aligned}
\]
as in \cite[proof of Theorem 8.2]{Kuhnt}. Moreover, write \(
M_{1,\mathfrak Q}
=M^0_{1,\mathfrak Q}L_{\mathfrak Q},
\) where $M^0_{1,\mathfrak Q}/k_{\mathfrak q}$ is abelian with elementary abelian inertia group. Indeed, let $E_{\mathfrak q}$ denote the completion of $E_{\mathbf n,\mathbf n^-}$ inside $M_{1,\mathfrak Q}$. Put \(
M^0_{1,\mathfrak Q}
=N^0_{\mathfrak Q}E_{\mathfrak q}.
\) Then \(
M_{1,\mathfrak Q}
=M^0_{1,\mathfrak Q}L_{\mathfrak Q}.
\) Since $E_{\mathfrak q}/k_{\mathfrak q}$ is elementary abelian, $M^0_{1,\mathfrak Q}/k_{\mathfrak q}$ is abelian, and $M^0_{1,\mathfrak Q}/L_{\mathfrak Q}^{\nr}$ is elementary abelian, see in \Cref{fig:truncation-field-diagram}.

\medskip
\noindent\textbf{Step 2. Remove the highest ramification layer of $M_1/L$.}
Let
\[
H=\Gal(M^0_{1,\mathfrak Q}/L_{\mathfrak Q}^{\nr}),
\]
and let $H^{n_{\mathfrak q}-1}$ be its highest nontrivial upper ramification group. Since the ramification depth of $L_{\mathfrak Q}/k_{\mathfrak q}$ is less than $n_{\mathfrak q}$, the subgroup $H^{n_{\mathfrak q}-1}$ fixes $M^0_{1,\mathfrak Q}\cap L_{\mathfrak Q}$. Thus it identifies with a subgroup
\[
H_0\subseteq\Gal(M_1/L).
\]
Moreover the restriction of $H_0$ to $\Gal(E_{\mathbf n,\mathbf n^-}/k)$ is surjective. Indeed, its annihilator in \(
\Hom(\Gal(E_{\mathbf n,\mathbf n^-}/k),\mathbb F_p)
\) consists of characters trivial on the highest upper ramification group at $\mathfrak q$. These characters have ramification depth at most $n_{\mathfrak q}-1$, hence belong to $\Hom(\Gal(E_{\mathbf n,\mathbf n^-}/k),\mathbb F_p)\cap V_{\mathbf n^-}=0$. Therefore the annihilator is trivial, and the restriction map is surjective. By the local ramification estimate \cite[Chapter IV, Proposition 7]{Serrelocal} or \cite[Theorem 5.23]{Kuhnt}, we have
\[
d_p(H_0)\leq ef_{\mathfrak q},
\]
where $f_{\mathfrak q}=\deg\mathfrak q$. Since $\Gal(M_1/L)$ is central in $\Gal(M_1/k)$, the subgroup $H_0$ is normal in $\Gal(M_1/k)$. Put $M_2=M_1^{H_0}$. Then
\[
\begin{aligned}
d_p\Gal(M_2/L)
&=d_p\Gal(M_1/L)-d_p(H_0)\\
&\geq d_p\Gal(N/L)
+d_p(G_{S,\mathbf n}^{T_k})
-d_p(G_{S,\mathbf n^-}^{T_k})
-d_p(H_0).
\end{aligned}
\]
We claim that
\[
d_p\Gal(M_2/k)=d_p(G).
\]
Let $N^{\mathrm{ele}}/k$ be the maximal elementary abelian subextension of $N/k$. Since \(
\Gal(N/L)\subseteq\Phi(\Gal(N/k)),
\) we have $N^{\mathrm{ele}}\subseteq L$. Moreover, since $N\cap E_{\mathbf n,\mathbf n^-}=k$, the extension $N^{\mathrm{ele}}E_{\mathbf n,\mathbf n^-}/k$ is the maximal elementary abelian subextension of $M_1/k$. The relevant field extensions can be seen in \cref{fig:generator-rank-truncation}.

\begin{figure}[ht]
\centering
\[
\begin{tikzcd}[ampersand replacement=\&,cramped]
	\&\& {M_1} \&\& \\
	\\
	{M_2} \&\& N \&\& {LE_{\mathbf n,\mathbf n^-}} \\
	\&\& L \&\& {N^{\mathrm{ele}}E_{\mathbf n,\mathbf n^-}} \\
	\&\& {N^{\mathrm{ele}}} \&\& {E_{\mathbf n,\mathbf n^-}} \\
	\&\& k
	\arrow["{H_0}", no head, from=3-1, to=1-3]
	\arrow[no head, from=3-1, to=4-3]
	\arrow[no head, from=3-3, to=1-3]
	\arrow[no head, from=3-3, to=4-3]
	\arrow[no head, from=3-5, to=1-3]
	\arrow[no head, from=4-3, to=3-5]
	\arrow[no head, from=4-3, to=5-3]
	\arrow[no head, from=4-5, to=3-5]
	\arrow[no head, from=5-3, to=4-5]
	\arrow[no head, from=5-3, to=6-3]
	\arrow[no head, from=5-5, to=4-5]
	\arrow[no head, from=6-3, to=5-5]
\end{tikzcd}
\]
\caption{The global field diagram for \cref{prop:Truncation of high ramification}.}
\label{fig:generator-rank-truncation}
\end{figure}

Since $H_0\subseteq\Gal(M_1/L)$ restricts surjectively to $\Gal(E_{\mathbf n,\mathbf n^-}/k)$, Galois correspondence gives \(
M_2\cap LE_{\mathbf n,\mathbf n^-}=L.
\) Consequently, \(
M_2\cap N^{\mathrm{ele}}E_{\mathbf n,\mathbf n^-}
=N^{\mathrm{ele}}.
\) Thus $N^{\mathrm{ele}}/k$ is also the maximal elementary abelian subextension of $M_2/k$. It follows that
\[
d_p\Gal(M_2/k)
=d_p\Gal(N/k)
=d_p(G).
\]
Finally, since $M_{1,\mathfrak Q}=M^0_{1,\mathfrak Q}L_{\mathfrak Q}$, we have \(
M_{2,\mathfrak Q}
=M_{1,\mathfrak Q}^{H_0}
=L_{\mathfrak Q}
\bigl(M^0_{1,\mathfrak Q}\bigr)^{
H_0|_{M^0_{1,\mathfrak Q}}}.
\) Put \(
M^0_{2,\mathfrak Q}
=
\bigl(M^0_{1,\mathfrak Q}\bigr)^{
H_0|_{M^0_{1,\mathfrak Q}}}.
\) Then \(
M_{2,\mathfrak Q}
=M^0_{2,\mathfrak Q}L_{\mathfrak Q},
\) where $M^0_{2,\mathfrak Q}/k_{\mathfrak q}$ is abelian with elementary abelian inertia group.  Since \(N^0_{\mathfrak Q}/L_{\mathfrak Q}^{\nr}\) is an elementary abelian extension, its upper ramification jumps are prime to \(p\) by \cref{lem:jumps-prime-to-p}. Hence, $ p\nmid (n_{\mathfrak q}-1) \text{ and  }     \left\lfloor
        \frac{n_{\mathfrak q}-1}{p}\right\rfloor
        =\left\lfloor\frac{n_{\mathfrak q}-2}{p}\right\rfloor.$   By \cref{thm:generator rank}, it follows that
\[d_p(\Gal(M_2/L))
    \geq d_p(\Gal(N/L))+
             d_p(\Delta_{S,\mathbf n}^{T_k})
        -d_p(\Delta_{S,\mathbf n^{-}}^{T_k}).\]
By construction, $M_2$ has a strictly lower ramification depth than $N$ at $\mathfrak Q$, while the ramification depths at all other places do not increase.

\medskip
\noindent\textbf{Step 3. Remove ramification of $N/L$ exceeding $\nu_L$ and add back elementary abelian extensions of $k$ with ramification depth bounded by $\nu$.}
Repeating this process with $M_2$ in place of $N$, until all ramification exceeding the bound $\nu_L$ has been removed, we obtain an extension $N_1/L$. Using the filtration
\[
\Delta_S^{T_k}
\subseteq\cdots\subseteq
\Delta_{S,\mathbf n}^{T_k}
\subseteq\Delta_{S,\mathbf n^-}^{T_k}
\subseteq\cdots\subseteq
\Delta_{S,\nu}^{T_k},
\]
the preceding inequalities to give
\[
d_p\Gal(N_1/L)
\geq d_p\Gal(N/L)
-\bigl(
d_p(\Delta_{S,\nu}^{T_k})
-d_p(\Delta_S^{T_k})
\bigr).
\]
Moreover $d_p(N_1/k)=d_p(G)$ by the above claim.  Let $E_{S,\nu}^{T_k}/k$ be the compositum of all elementary abelian $p$-extensions of $k$ with ramification depth bounded by $\nu$ in which all places of $T_k$ split completely. Put \(
M=N_1E_{S,\nu}^{T_k}.
\) Since $N_1\cap E_{S,\nu}^{T_k}/k$ is elementary abelian, we have
\[
d_p\Gal(N_1\cap E_{S,\nu}^{T_k}/k)
\leq d_p\Gal(N_1/k)=d_p(G).
\]
Consequently,
\[
\begin{aligned}
d_p\Gal(M/L)
&=d_p\Gal(N_1/L)
  +d_p\Gal(E_{S,\nu}^{T_k}N_1/N_1)\\
&=d_p\Gal(N_1/L)
  +d_p\Gal(E_{S,\nu}^{T_k}/k)
-d_p\Gal(N_1\cap E_{S,\nu}^{T_k}/k)\\
&\geq d_p\Gal(N_1/L)
  +d_p(G_{S,\nu}^{T_k})-d_p(G)\\
&\geq d_p\Gal(N/L)
  +d_p(\Delta_S^{T_k})
  -d_p(\Delta_{S,\nu}^{T_k})
+d_p(G_{S,\nu}^{T_k})-d_p(G).
\end{aligned}
\]
Finally, the ramification of \(M/L\) is still bounded by \(\nu_L\) by \cite[Lemma 2.3.7]{HXrationalbook}. By our construction, $M/L$ is central over $L/k$.

\begin{figure}[ht]

\centering
\[
\begin{tikzcd}
    &&&&& {M_{1,\mathfrak Q}} \\
    &&&& {N_{\mathfrak Q}} \\
    &&& {L_\mathfrak Q} \\
    && {L_\mathfrak Q^{ab}} \\
    && {L_\mathfrak Q^{\mathrm{ele}}M_\mathfrak Q^0} \\
    & {M_\mathfrak Q^{0,r}} &&&& {M^0_{1,\mathfrak Q}} \\
    {L_\mathfrak Q^{\mathrm{ele}}} && {L_\mathfrak Q^{\mathrm{nr}}M_\mathfrak Q^0} && {N^0_{\mathfrak Q}} \\
    &&& {L_\mathfrak Q^{\mathrm{nr}}} \\
    &&& {k_{\mathfrak{q}}}
    \arrow["{ele.ab}", no head, from=2-5, to=1-6]
    \arrow["{ele.ab}", no head, from=3-4, to=2-5]
    \arrow[no head, from=4-3, to=3-4]
    \arrow["unr"', no head, from=6-2, to=5-3]
    \arrow[no head, from=6-6, to=1-6]
    \arrow[bend left=16, no head, from=7-1, to=4-3]
    \arrow["{tot.ram}"', no head, from=7-1, to=6-2]
    \arrow[no head, from=7-3, to=5-3]
    \arrow[no head, from=7-5, to=2-5]
    \arrow["{ele.ab}"', no head, from=7-5, to=6-6]
    \arrow[draw=none, from=8-4, to=3-4]
    \arrow[no head, from=8-4, to=3-4]
    \arrow["{ele.ab}", no head, from=8-4, to=7-1]
    \arrow["{ele.ab}"', no head, from=8-4, to=7-3]
    \arrow["{ele.ab}"', no head, from=8-4, to=7-5]
    \arrow[no head, from=9-4, to=8-4]
\end{tikzcd}
\]
{\setlength{\abovecaptionskip}{1pt}%
\caption[Field diagram for \cref{prop:Truncation of high ramification}]
{The local field diagram for \cref{prop:Truncation of high ramification}.}
\label{fig:truncation-field-diagram}
}

\end{figure}

\item The extension $M/L$ has bounded depth, but it may still be ramified over $L$.  We now remove this ramification and estimate the rank lost in the process. For $\mathfrak q\in S$, we fix a place $\mathfrak Q$ above $q$ in $L$. Let $L_\mathfrak Q^{\mathrm{ab}}/k_\mathfrak q$ be the maximal abelian subextension of the local extension $L_\mathfrak Q/k_\mathfrak q$. Notice that $L_\mathfrak Q^{\mathrm{nr}}$ is also the maximal unramified subextension of $L_\mathfrak Q^{ab}/k_\mathfrak q.$ Let $I_\mathfrak q^{\mathrm{ab}}$ be its inertia group and let $L_\mathfrak Q^{ele}$ be the extension of $L_\mathfrak Q^{\mathrm{nr}}$ corresponding to the elementary abelian inertia quotient $I_\mathfrak q^{\mathrm{ab}}/(I_\mathfrak q^{\mathrm{ab}})^p$. By the construction of $M$ and the local field $M_\mathfrak Q^0,$ one sees that $M_\mathfrak Q^0L_\mathfrak Q^{\mathrm{nr}}/L_\mathfrak Q^{\mathrm{nr}}$ is an elementary abelian extension. Then $M_\mathfrak Q^0 L_\mathfrak Q^{ele}/L_\mathfrak Q^{\mathrm{nr}}$ is also an elementary abelian extension. Since \(L_{\mathfrak Q}^{\mathrm{ele}}M_{\mathfrak Q}^0/L_{\mathfrak Q}^{\mathrm{ele}}\) is elementary abelian, we may choose a complement to its inertia subgroup. The fixed field \(M_{\mathfrak Q}^{0,r}\) of this complement is then maximal totally ramified over \(L_{\mathfrak Q}^{\mathrm{ele}}\). See \cref{fig:truncation-field-diagram} for the corresponding field extension diagram. Since the ramification depth of $L_{\mathfrak Q}/k_{\mathfrak q}$ is at most $\nu_{\mathfrak q}$, the quotient formula gives
\[
(\Gal(L_{\mathfrak Q}/k_{\mathfrak q}))^{\nu_{\mathfrak q}}
=(\Gal(M_{\mathfrak Q}/k_{\mathfrak q}))^{\nu_{\mathfrak q}}\Gal(M_{\mathfrak Q}/L_{\mathfrak Q})/\Gal(M_{\mathfrak Q}/L_{\mathfrak Q})=1.
\]
Thus $(\Gal(M_{\mathfrak Q}/k_{\mathfrak q}))^{\nu_{\mathfrak q}}\subseteq \Gal(M_{\mathfrak Q}/L_{\mathfrak Q})$. By the subgroup formula and the ramification bound on $M_{\mathfrak Q}/L_{\mathfrak Q}$, we obtain
\[
\begin{aligned}
\Gal(M_{\mathfrak Q}/k_{\mathfrak q})^{\nu_{\mathfrak q}}
=\Gal(M_{\mathfrak Q}/k_{\mathfrak q})^{\nu_{\mathfrak q}}
\cap \Gal(M_{\mathfrak Q}/L_{\mathfrak Q})
=\Gal(M_{\mathfrak Q}/L_{\mathfrak Q})^{
\psi_{L_{\mathfrak Q}/k_{\mathfrak q}}(\nu_{\mathfrak q})}
=1.
\end{aligned}
\]
Thus the quotient formula yields
\[
\begin{aligned}
\Gal(M_{\mathfrak Q}^{0,r}/k_{\mathfrak q})
^{\nu_{\mathfrak q}}
=
(\Gal(M_{\mathfrak Q}/k_{\mathfrak q}))^{\nu_{\mathfrak q}}\Gal(M_{\mathfrak Q}/M_{\mathfrak Q}^{0,r})/\Gal(M_{\mathfrak Q}/M_{\mathfrak Q}^{0,r})
=1.
\end{aligned}
\]
Hence $M_{\mathfrak Q}^{0,r}/k_{\mathfrak q}$ has ramification depth at most $\nu_{\mathfrak q}$. It is also abelian with elementary abelian inertia group $\Gal(M_\mathfrak Q^{0,r}/L_\mathfrak Q^{\mathrm{nr}})$. By \cref{lem:inertia-rank-bound}, we obtain
\begin{equation}
\label{eq:local-inertia-rank-step3}
   d_p(\Gal(M_\mathfrak Q^{0,r}/L_\mathfrak Q^{\mathrm{nr}}))
   \le
   e f_\mathfrak q\bigl(\nu_\mathfrak q-1-\floor{(\nu_\mathfrak q-1)/p}\bigr).
\end{equation}
Moreover, we have
\begin{equation}
\label{eq:local-drop-step3}
   d_p(\Gal(M_\mathfrak Q^{0,r}/L_\mathfrak Q^{ele}))
   =d_p(\Gal(M_\mathfrak Q^{0,r}/L_\mathfrak Q^{\mathrm{nr}}))-d_p(I_q^{\mathrm{ab}}/(I_q^{\mathrm{ab}})^p).
\end{equation}
By the definition of $S_{\mathrm{nr}}$ in \cref{sec:relation-detailed}, we have
\begin{equation}
\label{eq:Ip-p-rank}
   d_p(I_\mathfrak q^{\mathrm{ab}}/(I_\mathfrak q^{\mathrm{ab}})^p)
   =
   \begin{cases}
      d_p(G_\mathfrak q)-1, & \text{ if }\mathfrak q\in S_{\mathrm{nr}};\\
      d_p(G_\mathfrak q),   & \text{ if } \mathfrak q\in S\setminus S_{\mathrm{nr}}.
   \end{cases}
\end{equation}
For each $\mathfrak q\in S$, choose a place $\mathfrak Q$ above $\mathfrak q$, and let $I_{\mathfrak Q}$ be the inertia group of $M_{\mathfrak Q}/L_{\mathfrak Q}$. By the construction of $M_{\mathfrak Q}^{0,r}$,we have
\[
d_p(I_{\mathfrak Q})
\le
d_p\Gal(M_{\mathfrak Q}^{0,r}/
L_{\mathfrak Q}^{\mathrm{ele}}).
\]
Since $\Gal(M/L)$ is central in $\Gal(M/k)$, the inertia groups above each $\mathfrak q\in S$ coincide. Let $J=\langle I_{\mathfrak Q}:\mathfrak q\in S\rangle$ and $\tilde L=M^J.$ Then $\tilde L/L$ is unramified. Since $M/L$ is elementary abelian, combining Equations \eqref{eq:local-inertia-rank-step3}, \eqref{eq:local-drop-step3}, and \eqref{eq:Ip-p-rank}, we obtain
\begin{align*}
   d_p(\Gal(\tilde{L}/L))
   &=d_p\Gal(M/L)-d_p(J) \ge d_p\Gal(M/L)
-\sum_{\mathfrak q\in S}d_p(I_{\mathfrak Q})\\
   &\ge d_p(\Gal(M/L))
      -\sum_{\mathfrak q\in S}
       d_p(\Gal(M_\mathfrak Q^{0,r}/L_\mathfrak Q^{ele})) \notag\\
   &\ge d_p(\Gal(M/L))
      -\sum_{\mathfrak q\in S} e f_\mathfrak q\bigl(\nu_\mathfrak q-1-\floor{(\nu_\mathfrak q-1)/p}\bigr)
      +\sum_{\mathfrak q\in S}d_p(I_\mathfrak q^{\mathrm{ab}}/(I_\mathfrak q^{\mathrm{ab}})^p)\\
        &\ge   d_p(\Gal(N/L))
        +d_p(\Delta_S^{T_k})
        +1-|{T_k}|
        -d_p(G)
      +\sum_{\mathfrak q\in S} d_p(G_\mathfrak q)-|S_{nr}|.
\label{eq:N0-initial}
\end{align*}

\end{enumerate}
\end{proof}

Now we can give the complete proof of \cref{thm:kuhnt}.

\begin{proof}[Proof of \cref{thm:kuhnt}]
Assume that $L/k$ is finite.  We show that if $\delta_N>0$, then there is a further nontrivial unramified $p$-extension of $L$ in which the places of $T_L$ split completely.  Hence, if $L$ is maximal with that property, then $\delta_N\le 0$. 

Consider the map as follows:
\[
\rho:
H^2(G,\Z/p\Z)
\longrightarrow
\prod_{\gp\in S\setminus S_{\nr}}
H^2(G_\gp,\Z/p\Z)
\times
\prod_{\gp\in S_{\nr}}
\frac{H^2(G_\gp,\Z/p\Z)}
{H^2_{\nr}(G_\gp,\Z/p\Z)},
\]
for each \(\gp\in S\setminus S_{\nr}\), the corresponding component of \(\rho\) is the restriction map \(\res^G_{G_\gp},\) and for \(\gp\in S_{\nr}\), it is the composition of \(\res^G_{G_\gp}\) with the canonical projection onto \(H^2(G_\gp,\Z/p\Z)/H^2_{\nr}(G_\gp,\Z/p\Z)\). By \cref{prop:local-hnr}, the kernel of $\rho$ is $H^2_{\nr}(G,\Z/p\Z)$.  In \cref{sec:relation-detailed}, we give the induced filtration of $H^2_{nr}(G_{\mathfrak{p}},\Z/p\Z)$  by  $\operatorname{Fil}_NH^2_{nr}(G_{\mathfrak{p}},\Z/p\Z).$ The similar filtration can be induced in the global case by
\[
        \operatorname{Fil}_NH^2_{nr}(G,\Z/p\Z):=H^2_{nr}(G,\Z/p\Z)\cap \operatorname{Fil}_NH^2_{}(G,\Z/p\Z).
\]
The quotient filtration of \(H^2(G_{\mathfrak{p}},\Z/p\Z)/H^2_{nr}(G_{\mathfrak{p}},\Z/p\Z)\) can be given by
\[
        \operatorname{Fil}_N(H^2(G_{\mathfrak{p}},\Z/p\Z)/H^2_{nr}(G_{\mathfrak{p}},\Z/p\Z))
        :=
        (\operatorname{Fil}_NH^2(G_{\mathfrak{p}},\Z/p\Z)+H^2_{nr}(G_{\mathfrak{p}},\Z/p\Z))/H^2_{nr}(G_{\mathfrak{p}},\Z/p\Z).
\]
The isomorphism theorem gives the following equation $$\dim_{\Fp}\operatorname{Fil}_N(H^2(G_{\mathfrak{p}},\Z/p\Z)/H^2_{nr}(G_{\mathfrak{p}},\Z/p\Z))=\dim_{\Fp} \operatorname{Fil}_NH^2(G_{\mathfrak{p}},\Z/p\Z)-  \dim_{\Fp}\operatorname{Fil}_NH^2_{nr}(G_{\mathfrak{p}},\Z/p\Z).$$
Since the restriction maps are filtered by  \cref{Zassenhaus map}, then \(\rho\) can be filtered and
\[\rho(\operatorname{Fil}_NH^2(G,\Z/p\Z))
        \subseteq
        \prod_{\mathfrak p\in S\setminus S_{\nr}}\operatorname{Fil}_NH^2(G_{\mathfrak{p}},\Z/p\Z)
        \times
        \prod_{\mathfrak p\in S_{\nr}} \operatorname{Fil}_N(H^2(G_{\mathfrak{p}},\Z/p\Z)/H^2_{nr}(G_{\mathfrak{p}},\Z/p\Z)).\]
Denoted the above map by \(\rho_{fil}\). One can check that  $$\operatorname{Fil}_NH^2_{nr}(G,\Z/p\Z)=\ker(\rho_{fil}).$$
Now use the dimension formulas in \cref{prop:fildual}, we have
\[
\dim_{\Fp}\operatorname{Fil}_NH^2(G,\Z/p\Z)
=
\sum_{n=2}^{N}s_n(G)\quad \text{ and } \quad \dim_{\Fp}\operatorname{Fil}_NH^2(G_{\mathfrak{p}},\Z/p\Z)
=
\sum_{n=2}^{N}s_n(G_{\mathfrak{p}}).
\]
Moreover, since $\dim_{\Fp}\operatorname{Fil}_NH^2_{nr}(G_{\mathfrak{p}},\Z/p\Z)=\sum_{n=2}^{N}u_n(G_{\mathfrak{p}})\leq \sum_{n=2}^{N}s_n(G_{\mathfrak{p}})$, we get
$$\dim_{\F_p}\operatorname{Fil}_NH^2_{nr}(G,\Z/p\Z)\geq
\sum_{n=2}^{N}s_n(G)
-\sum_{\mathfrak p\in S\setminus S_{\nr}}
\sum_{n=2}^{N}s_n(G_{\mathfrak{p}})
-
\sum_{\mathfrak p\in S_{\nr}}\sum_{n=2}^{N}[s_n(G_{\mathfrak{p}})-u_n(G_{\mathfrak{p}})]
$$
by considering the filtered map $\rho_{fil}$.

Choose a basis $\epsilon_1,\ldots,\epsilon_n$ of $\operatorname{Fil}_NH^2_{nr}(G,\Z/p\Z)$.  By \cref{prop:global-embedding-solvable}, each $\epsilon_i$ gives a proper global solution $M_i/k$ with kernel $\Z/p\Z$.  Let $Y=Y_1\cdots Y_n.$
The $G$-action on the kernel $\Z/p\Z$ is trivial, so $\Gal(Y/L)$ is central in $\Gal(Y/k)$.  By \cref{Marked solution}, we have
\begin{equation*}
   d_p(\Gal(Y/L))=\dim_{\F_p}\operatorname{Fil}_NH^2_{nr}(G,\Z/p\Z).
\end{equation*}
Thus, we obtain
\begin{equation*}
   d_p(\Gal(Y/L))\ge \sum_{n=2}^{N}s_n(G)
-\sum_{\mathfrak p\in S\setminus S_{\nr}}
\sum_{n=2}^{N}s_n(G_{\mathfrak{p}})
-
\sum_{\mathfrak p\in S_{\nr}}\sum_{n=2}^{N}(s_n(G_{\mathfrak{p}})-u_n(G_{\mathfrak{p}})).
\end{equation*}
By \cref{prop:Truncation of high ramification}, there exists an extension $\tilde{L}/L$ such that
\begin{align*}
   d_p(\Gal(\tilde{L}/L))
&\ge   \sum_{n=2}^{N}s_n(G)-d_p(G)
-\sum_{\mathfrak p\in S\setminus S_{\nr}}
\left(\sum_{n=2}^{N}s_n(G_{\mathfrak{p}})-d_p(G_{\mathfrak{p}})\right)\\
 &\qquad-\sum_{\gp\in S_{\mathrm{nr}}}\left(\sum_{n=2}^{N}(s_n(G_{\mathfrak{p}})-u_n(G_{\mathfrak{p}})) -d_p(G_\mathfrak{p})+1\right)
         +d_p(\Delta_S^{T_k})+1-|{T_k}|.
\end{align*}
\end{proof}

\section{Ray class fields and the tower criterion}\label{sec:5}
In this section, we provide a complete proof of the main tower criterion \cref{thm:maincriterion}.

\subsection{Local relation module contribution}
 In the sense of \cref{thm:kuhnt}, for each \(\mathfrak p\in S\), we need to bound the local contributions
\[
\sum_{n=2}^{N}s_n(G_{\mathfrak p})-d_p(G_{\mathfrak p}) \text{ and }\sum_{n=2}^{N}(s_n(G_{\mathfrak{p}})-u_n(G_{\mathfrak{p}})) -d_p(G_\mathfrak{p})+1.
\]
We begin by introducing some notations. For \(A=G_{\mathfrak p}\) or \(I_{\mathfrak p}\), let \(s_n(A)\) denote the number of independent defining relations of Zassenhaus degree \(n\) in the standard minimal pro-\(p\) presentation of the finite abelian \(p\)-group \(A\), as defined in \cref{Defining relation}. By \cref{prop:local-hnr}, the dimension of the \(n\)-th graded piece of the unramified local cohomology is \(u_n(G_{\mathfrak p})\). Now we explain how to bound the local contribution for each \(\mathfrak p\in S\). 

\begin{proposition}\label{prop:local-relation}
Under the notation of \cref{sec:relation-detailed}, let $\gp\in S$ and assume that $G_{\mathfrak p}$ is a finite abelian $p$-group, with inertia subgroup $I_{\mathfrak p}$. Then for every $N\ge2$, we have
\[
\sum_{n=2}^{N}s_n(G_{\mathfrak p})
-d_p(G_{\mathfrak p})
\le
\sum_{n=2}^{N}s_n(I_{\mathfrak p}).
\]
Moreover, if $G_{\mathfrak p}$ is nontrivial and $I_{\mathfrak p}$ is an elementary abelian $p$-group, then
\[
\sum_{n=2}^{N}
(s_n(G_{\mathfrak p})-u_n(G_{\mathfrak p}))
-d_p(G_{\mathfrak p})+1
\le
\sum_{n=2}^{N}s_n(I_{\mathfrak p}).
\]
\end{proposition}

\begin{proof}
If $G_{\mathfrak p}=1$, the first inequality is immediate. Hence assume that $G_{\mathfrak p}\ne1$. Since \(G_{\mathfrak p}\) is a finite abelian \(p\)-group, the exact sequence
\begin{equation}
\label{equ:local sequence}
      1\longrightarrow I_\gp\longrightarrow G_\gp\longrightarrow \Z/p^{a_\gp}\Z\longrightarrow 1
\end{equation}
gives us the inequality $d_p(I_{\mathfrak p})\leq d_p(G_{\mathfrak p})\le d_p(I_{\mathfrak p})+1.$ For a finite abelian \(p\)-group  $        A\simeq
        \bigoplus_{i=1}^{a}\mathbb Z/p^{\alpha_i}\mathbb Z$ with $ 
        \alpha_i\ge1.$ Let $N\ge 2$ and $\ell=\lfloor\log_p N\rfloor.$ Since \(
d_p(p^\ell A)=\#\{i:\alpha_i>\ell\},
\) then by \cref{prop:abelian-standard-relations}, we obtain
\begin{equation}\label{eq:5.2}
\sum_{n=2}^{N}s_n(A)
=
\binom{d_p(A)}{2}
+d_p(A)-d_p(p^\ell A).
\end{equation}
Moreover, the inclusion $p^\ell I_{\mathfrak p}\subseteq p^\ell G_{\mathfrak p}$ implies $d_p(p^\ell I_{\mathfrak p})
\le d_p(p^\ell G_{\mathfrak p})$. Applying \cref{eq:5.2}, we obtain
\[
\begin{aligned}
\sum_{n=2}^{N}s_n(G_{\mathfrak p})
-d_p(G_{\mathfrak p})
&=
\binom{d_p(G_{\mathfrak p})}{2}
-d_p(p^\ell G_{\mathfrak p})
\le
\binom{d_p(I_{\mathfrak p})+1}{2}
-d_p(p^\ell I_{\mathfrak p})\\
&=
\sum_{n=2}^{N}s_n(I_{\mathfrak p}).
\end{aligned}
\]
This proves the first claim.

Now assume that  $I_{\mathfrak p}$ is an elementary abelian $p$-group. From the local exact sequence \cref{equ:local sequence}, one can write  $I_{\mathfrak p}$  and $G_\gp$ as follows $$I_{\mathfrak p}
        \simeq
        \bigoplus_{i=1}^{h}\mathbb Z/p\mathbb Z \quad \text{ and } \quad
        G_\gp\simeq   \begin{cases}
      \Z/p^{a_\gp}\Z\bigoplus(\mathbb Z/p\mathbb Z)^h, & \text{ if }\mathfrak{p}\in S_{nr};\\
    \Z/p^{a_\gp+1}\Z\bigoplus(\mathbb Z/p\mathbb Z)^{h-1}, & \text{ if } \mathfrak{p}\notin S_{nr}.\\
   \end{cases}
    $$
If $\mathfrak p\in S_{\mathrm{nr}}$, then $a_{\mathfrak p}\ge1$, since otherwise $G_{\mathfrak p}=I_{\mathfrak p}$. For $\mathfrak{p}\in S_{nr}$, $G_\gp/I_\gp\cong \Z/p^{a_\gp}\Z$  has a unique defining relation $\bar\sigma^{p^{a_\gp}}=1$. Let \(\mathbf{1}_{\{n=m\}}\) and \(\mathbf{1}_{\{n\ge m\}}\) denote the indicator functions of the conditions \(n=m\) and \(n\ge m\), respectively; that is, each is equal to \(1\) when the corresponding condition holds and \(0\) otherwise. By \cref{prop:abelian-standard-relations}, we have
\[s_n(G_{\mathfrak p})
=\binom{h+1}{2}\mathbf 1_{\{n=2\}}+h\cdot \mathbf 1_{\{n=p\}}+\mathbf 1_{\{n=p^{a_{\mathfrak p}}\}} \quad \text{ and } \quad
s_n(G_{\mathfrak p}/I_{\mathfrak p})=\mathbf 1_{\{n=p^{a_{\mathfrak p}}\}}.\]
Since \(
\dim_{\Fp} H^2_{\mathrm{nr}}(G_\gp,\mathbb Z/p\mathbb Z)
=
1
=
\dim_{\Fp} H^2(G_\gp/I_\gp,\mathbb Z/p\mathbb Z)
\) by \cref{prop:local-hnr}, it follows from \cref{eq:unramified-cohomology} that inflation induces an isomorphism
\[
H^2(G_\gp/I_\gp,\mathbb Z/p\mathbb Z)
\cong
H^2_{\mathrm{nr}}(G_\gp,\mathbb Z/p\mathbb Z).
\]
Then the unique defining relation of $G_\gp/I_\gp$ is  $\bar{\sigma}^{p^a_{\gp}}$  where $\sigma$ generates the cyclic direct summand of $G_\gp\simeq\langle\sigma\rangle\oplus I_\gp$. We claim that $\sigma^{p^{a_{\gp}}}$$= 1$  can be viewed as the unique unramified define relation of  $G_\gp$. Indeed, let \(
\eta\in
H^2(G_{\mathfrak p}/I_{\mathfrak p},\mathbb Z/p\mathbb Z)
\) be a nonzero generator. Since \(
G_{\mathfrak p}/I_{\mathfrak p}
\simeq C_{p^{a_{\mathfrak p}}},
\) the class $\eta$ is represented by the non-split central extension
\[
1\longrightarrow \langle z\rangle
\longrightarrow
\langle \widetilde{\sigma}\rangle
\longrightarrow
\langle\bar{\sigma}\rangle
\longrightarrow1,
\]
where \(
|\widetilde{\sigma}|=p^{a_{\mathfrak p}+1},
|z|=p\text{ and }
\widetilde{\sigma}^{\,p^{a_{\mathfrak p}}}=z.
\)  Thus the relation \(
\bar{\sigma}^{\,p^{a_{\mathfrak p}}}=1
\) in $G_{\mathfrak p}/I_{\mathfrak p}$ lifts to the nontrivial kernel element $z$. Now inflate $\eta$ along the natural projection \(
\pi:G_{\mathfrak p}\longrightarrow G_{\mathfrak p}/I_{\mathfrak p}.
\) The inflated class is represented by the central extension similar to \cref{Diagram:inflation}:
\[
1\longrightarrow \langle z\rangle
\longrightarrow
\widetilde G_{\mathfrak p}
\longrightarrow
G_{\mathfrak p}
\longrightarrow1.
\]
Since \(
G_{\mathfrak p}
\simeq
\langle\sigma\rangle\oplus I_{\mathfrak p},
\) one can find that  \(
\widetilde G_{\mathfrak p}
\simeq
\langle\widetilde{\sigma}\rangle\oplus I_{\mathfrak p},
\) and the lift of $\sigma$ satisfies \(
\widetilde{\sigma}^{\,p^{a_{\mathfrak p}}}=z.
\) Thus, by the naturality of transgression and \cref{lem:relation mod dual}, the inflated class is dual to the power relation $\sigma^{p^{a_{\mathfrak p}}}=1$ in the standard minimal presentation of $G_{\mathfrak p}$. By \cref{prop:abelian-standard-relations,prop:fildual}, this class has filtration degree exactly $p^{a_{\mathfrak p}}$. Hence $\sigma^{p^{a_{\mathfrak p}}}=1$ may be chosen as the unique unramified defining relation in this  presentation, and \(
u_n(G_{\mathfrak p})
=\mathbf{1}_{\{n=p^{a_{\mathfrak p}}\}}
=s_n(G_{\mathfrak p}/I_{\mathfrak p}).
\) Thus we obtain
\[
\begin{aligned}
\sum_{n=2}^{N}
\bigl(s_n(G_{\mathfrak p})-u_n(G_{\mathfrak p})\bigr)
-d_p(G_{\mathfrak p})+1
& =\binom{h+1}{2}
+h\cdot \mathbf{1}_{\{N\ge p\}}
+\mathbf{1}_{\{N\ge p^{a_{\mathfrak p}}\}}
-\mathbf{1}_{\{N\ge p^{a_{\mathfrak p}}\}}
-h                                    \\
&=\binom{h}{2}+h\cdot \mathbf{1}_{\{N\ge p\}}                   =\sum_{n=2}^{N}s_n(I_{\mathfrak p}).
\end{aligned}
\]
If  $\mathfrak{p}\notin S_{nr}$, then $h\ge1$ and $H_{\mathrm{nr}}^2(G_{\mathfrak p},\mathbb Z/p\mathbb Z)=0$. Similarly, we can get the following equations
\[
s_n(G_{\mathfrak p})
=
\binom{h}{2}\mathbf{1}_{\{n=2\}}
+
(h-1)\cdot \mathbf{1}_{\{n=p\}}
+
\mathbf{1}_{\{n=p^{a_{\mathfrak p}+1}\}} \quad \text{ and } \quad
u_n(G_{\mathfrak p})=0.
\]
Moreover, we can deduce the estimate as follows
\[
\begin{aligned}
    \sum_{n=2}^{N}(s_n(G_{\mathfrak{p}})-u_n(G_{\mathfrak{p}})) -d_p(G_\mathfrak{p})+1 &=\binom{h}{2}+(h-1)\cdot \textbf{1}_{\{N\geq p\}}+\textbf{1}_{\{N\geq p^{a_\gp+1}\}}-h+1\\
     &\leq \binom{h}{2}+h\cdot \textbf{1}_{\{N\geq p\}}=\sum_{n=2}^{N}s_n(I_{\mathfrak p}).
\end{aligned}
\]
\end{proof}

\begin{remark}
    Suppose that $K/k$ is abelian in \cref{thm:kuhnt}. Then \(G_{\mathfrak p}=\operatorname{Gal}(L_{\mathfrak P}/k_{\mathfrak p})\) is abelian. Indeed, since $K/k$ is abelian, the commutator subgroup of $G_{\mathfrak p}$ is contained in $\operatorname{Gal}(L_{\mathfrak P}/K_{\mathfrak q})$. It is also contained in the inertia subgroup of $G_{\mathfrak p}$. As $L/K$ is unramified, the intersection of these two subgroups is trivial. Hence $G_{\mathfrak p}$ is abelian.
\end{remark}

\subsection{Conductor-discriminant formula and the p-rank of ray class field}
Before proving the main theorem, we recall several classical results that will be used in the proof of \cref{thm:maincriterion}.
\begin{proposition}
\cite[Proposition 5.2]{Iharacompositio}
\label{prop:condgenus}
    Let $K/F$ be a geometric abelian extension of function fields over $\F_q$.  The conductor-discriminant formula says
\begin{equation}
\label{eq:conductor-discriminant}
   2g(K)-2=[K:F](2g(F)-2)+\sum_{\chi}\deg\mathfrak f_\chi,
\end{equation}
where $\chi$ ranges over the characters of $\Gal(K/F)$ and $f_\chi$ is the Artin conductor of $\chi$.

\end{proposition}

\begin{proposition}
\cite[Proposition 5.3]{Iharacompositio}
\label{prop:ray-rank}
    Let $k^T_\gm/k$ be the maximal elementary abelian pro-p extension of conductor $\gm=\sum_{\gp\in S}\nu_\gp\gp$ in which the places in $T$ split completely.  The $p$-rank of the Galois group $\Gal(k^T_\gm/k)$ has the following lower bound
\begin{equation*}
\label{eq:ray-rank}
   d_p(\Gal(k^T_\gm/k))
   \ge 1+
   \sum_{\gp\in S} e f_\gp\bigl(\nu_\gp-1-\floor{(\nu_\gp-1)/p}\bigr)-t.
\end{equation*}
\end{proposition}

\begin{corollary}
\label{cor:average-conductor}
Let $k^T_\gm/k$ be as in \cref{prop:ray-rank}, where $\nu_\gp\ge 2$ for all $\gp\in S$. Then
\[
\begin{aligned}
\frac{1}{[k^T_\gm:k]}
\sum_\chi \deg\mathfrak f_\chi
\le \sum_{\gp\in S} f_\gp
\left(
\nu_\gp-
\frac{
2+\displaystyle\sum_{j=2}^{\nu_\gp-1}
q^{f_\gp(j-1-\floor{(j-1)/p})}
}{
q^{f_\gp(\nu_\gp-1-\floor{(\nu_\gp-1)/p})}
}\right).
\end{aligned}
\]
\end{corollary}

\begin{proof}

Fix $\gp\in S$ and write $f_{\chi,\gp}=v_\gp(\mathfrak f_\chi)$. By \cref{prop:condgenus} and \cite[Theorem 1.4]{Auer2000}, we have
\[
\frac{1}{[k^T_\gm:k]}\sum_\chi f_{\chi,\gp}
=
\nu_\gp-
\sum_{j=0}^{\nu_\gp-1}
\frac{[k^T_{\gm-(\nu_\gp-j)\gp}:k]}
{[k^T_\gm:k]}.
\]
By the idelic exact sequence in the proof of \cref{thm:generator rank}, for $1\le j\le\nu_\gp$, the Galois group $\Gal(k^T_\gm/k^T_{\gm-(\nu_\gp-j)\gp})$ is a quotient of the group
\[
\frac{(U_\gp^{(1)})^pU_\gp^{(j)}}
{(U_\gp^{(1)})^pU_\gp^{(\nu_\gp)}}, \text{ with order } \frac{
q^{f_\gp(\nu_\gp-1-\floor{(\nu_\gp-1)/p})}
}{
q^{f_\gp(j-1-\floor{(j-1)/p})}
}.
\]
Since $k^T_{\gm-\nu_\gp\gp}
=k^T_{\gm-(\nu_\gp-1)\gp},$ we have
\[
\begin{aligned}
\sum_{j=0}^{\nu_\gp-1}
\frac{[k^T_{\gm-(\nu_\gp-j)\gp}:k]}
{[k^T_\gm:k]}
&=
\sum_{j=0}^{1}
\frac{[k^T_{\gm-(\nu_\gp-j)\gp}:k]}
{[k^T_\gm:k]}
+
\sum_{j=2}^{\nu_\gp-1}
\frac{[k^T_{\gm-(\nu_\gp-j)\gp}:k]}
{[k^T_\gm:k]}\\
&\ge
\frac{2}{
q^{f_\gp(\nu_\gp-1-\floor{(\nu_\gp-1)/p})}
}
+
\frac{
\displaystyle\sum_{j=2}^{\nu_\gp-1}
q^{f_\gp(j-1-\floor{(j-1)/p})}
}{
q^{f_\gp(\nu_\gp-1-\floor{(\nu_\gp-1)/p})}
}.
\end{aligned}
\]
The assertion follows by summing over $\gp\in S$.
\end{proof}

\subsection{The proof of the tower criterion \texorpdfstring{\protect\cref{thm:maincriterion}}{Theorem 1.1}}
Let $K=k^T_\gm$ be the maximal elementary abelian pro-p extension of $k$ with conductor $\gm=\sum_{\gp\in S}\nu_\gp\gp$ in which the places in $T$ split completely. Let $L/K$ be the maximal unramified $p$-extension in which the places above $T$ split completely.  Let $G=\Gal(L/k)$ and \(d=1+
   \sum_{\gp\in S}e f_\gp\bigl(\nu_\gp-1-\floor{(\nu_\gp-1)/p}\bigr)-t.
\) If $G$ is finite, then \cref{prop:ray-rank} gives
\begin{equation}\label{eq:d-lower}
   d_p(G)\ge d_p(\Gal(k^T_\gm/k)) \ge d.
\end{equation}

In the local sums of \cref{thm:kuhnt}, the places $\mathfrak p\in S$ with $G_{\mathfrak p}=1$ contribute zero. Applying \cref{prop:local-relation} to the remaining places, we obtain
\begin{align*}
\sum_{n=2}^N s_n(G)-d_p(G)
&\le
\sum_{\substack{\mathfrak p\in S\\G_{\mathfrak p}\ne1}}
\left(
\sum_{n=2}^N
\bigl(s_n(G_{\mathfrak p})-u_n(G_{\mathfrak p})\bigr)
-d_p(G_{\mathfrak p})+1
\right)
-1+|T_k|
\\
&\le
\sum_{\mathfrak p\in S}
\sum_{n=2}^N s_n(I_{\mathfrak p})
-1+|T_k|.
\end{align*}
Multiplying the above inequality by \((1-\tau)\tau^N\), summing over all \(N\ge 2\), and using \(0<\tau<1\), we obtain
\begin{align*}
\sum_{n\geq2}^{}s_n(G)\tau^n-d_p(G)\tau^2&\leq
 \left(
        |T_k|-1
\right)  \tau^2
+
\sum_{\mathfrak p\in S}
\sum_{n\ge 2}
s_n(I_{\mathfrak p})\tau^n.
\end{align*}
Since $L/K$ is unramified and $K/k$ is elementary abelian, the inertia group $I_{\mathfrak p}$ of $L/k$ is elementary abelian. Since $I_{\mathfrak p}\cong(C_p)^{d_p(I_{\mathfrak p})}$, \cref{prop:abelian-standard-relations} and \cref{lem:inertia-rank-bound} gives
\[
\sum_{n\ge2}s_n(I_{\mathfrak p})\tau^n
=
\binom{d_p(I_{\mathfrak p})}{2}\tau^2
+d_p(I_{\mathfrak p})\tau^p
\leq
\binom{n_{\mathfrak p}}{2}\tau^2
+n_{\mathfrak p}\tau^p.
\]

By the preceding inequalities and \eqref{eq:d-lower}, for any $\tau\in(0,1)$, the Golod--Shafarevich polynomial defined in \cref{def:GS-grp} satisfies
\begin{align*}
P_G(\tau)
&=1-d_p(G)\tau+\sum_{n\ge2}s_n(G)\tau^n\\
&\le 1-d_p(G)\tau+(d_p(G)+|T_k|-1)\tau^2
+\sum_{\mathfrak p\in S} \left( \binom{n_{\mathfrak p}}{2}\tau^2 +n_{\mathfrak p}\tau^p
\right)\\
&\le 1-d\tau+(d+|T_k|-1)\tau^2 +\sum_{\mathfrak p\in S}
\left(
\binom{n_{\mathfrak p}}{2}\tau^2
+n_{\mathfrak p}\tau^p
\right),
\end{align*}
where the last inequality follows from $-d_p(G)\tau(1-\tau)\le-d\tau(1-\tau)$ for any $\tau\in (0,1)$. In particular, if there exists \(\tau\in(0,1)\) such that the right-hand side is negative, then $P_G(\tau)<0.$
By the Golod--Shafarevich criterion \cref{thm:GS-infinite}, the pro-\(p\) group \(G\) is infinite. It remains to compute the asymptotic ratio.  Since the places above $T$ split completely in the tower, there are $[K:k]t$ rational places splitting completely over $K$.  Using Equation \eqref{eq:conductor-discriminant},
\[
   g(K)-1=[K:k]\left(g(k)-1+\frac{1}{2[K:k]}\sum_\chi\deg\mathfrak f_\chi\right).
\]
Therefore, Equation \eqref{eq:basic-bound} gives
\[
   A(q)\ge
   \frac{[K:k]t}{g(K)-1}
   =\frac{t}{\displaystyle g(k)-1+({1}/{2[K:k]})
      \sum_\chi\deg\mathfrak f_\chi}.
\]
By \cite[Corollary 1.2]{Iharacompositio}, the proof of \cref{thm:maincriterion} is complete.

\section{New lower bounds for Ihara constants over small prime finite fields}\label{sec:6}
In this section, we provide the new lower bounds for Ihara constants over small prime finite fields, such as $A(2)$ and $A(3)$.

\subsection{Ihara constants for small prime finite fields}
In \cite[Theorem 1.1]{Iharacompositio}, Duursma and Mak use the Golod--Shafarevich inequality to derive a ramification criterion for infinite class field towers. By incorporating the Zassenhaus filtration on the relation module, \cref{thm:maincriterion} refines their criterion and yields a sharper ramification condition. We then apply this refined criterion to obtain the improved  lower bounds for  Ihara constants over prime finite fields of small characteristics. 

\begin{theorem}\label{thm:6.1} The Ihara constant $A(2)$ is lower bounded by
    $$A(2)\geq 0.317779\cdots.$$
\end{theorem}
\begin{proof}
Let $H = \mathbb{F}_2(x,y)$ be the function field defined by $y^2 + (x^3 + x + 1)y + x^2 + x$ over $\F_2$. Let $a_d$ be the number of places of degree $d$.
By \cite[Proposition 5.5]{Iharacompositio}, the genus of $H$ is $g(H)=2$ and the sequence of number of places of $H$ with degree $d$ can be given by
\[
(a_d(H):d\geq1)=(6,0,1,1,6,13,\ldots).
\]
Set $P_3=(x^3+x+1,y+x^2)$ and $P_7=(x^7+x^3+1,y+x^5+x^2)$ be places of $H$ with degree $3$ and $7$ respectively. Let $k_h=H(z_h)$ defined by
$$z_h^2 + [(x^4+x^3+x^2+x)y +(x^2+x+1)]z_h = (x^2 + x)h,$$where $h\neq 0$ is in the span of the functions  $\{1,x,x^2,x^3,y\}$. We can check that every rational places of $H$ splits completely in $k_h$ for $h$ choosen above. Moreover, one can check that the conductor of $k_h/H$ is $2P_3+2P_7$ for $h\neq 1+x+y,1+x+x^3,y+x^3$ and the conductor of $k_h/H$ is $2P_7$ for $h=1+x+y,1+x+x^3,y+x^3.$

Let $k$ be the compositum of all the $k_h$, then $a_1(k)=192$. By \cref{prop:condgenus}, we have $$2g(k)-2=32\times (2g(H)-2)+28\times (2\times 10)+3\times (2\times 7).$$
Thus, the genus of $k$ is $g(k)=334$.

The place $P_3$ with degree $3$ splits completely in $k_{x^3+x+1}$ and is inertia in $ k_{1+x+y}$.  Then $k$ has at least $2$ places with degree $6$. The place $(x^2 + x+ 1)$ is a place with degree $4$ in $H$ which splits completely in the field $k_1/H$ and is inertia in $k_x/H$. Since the residue field is cyclic, we obtain $a_8(k)\geq 16.$ Moreover, one of the place over $(x^5+x^2+1)$ in $H$ splits completely in $k_{x^2}/H$ and is inertia in $k_{x^3}/H$, thus $a_{10}(k)\geq 16$. One of the place over $(x^6+x^3+1)$ in $H$ splits completely in $k_{1}/H$ and is inertia in $k_{x}/H$, thus $a_{12}(k)\geq 16$. Summarize the results as above, the sequence of number of places of $k$ with degree $d$ can be given by
\[
(a_d(k):d\geq1)=(192,0,0,0,0,\ge 2,1,16,0,\geq 16,0,\geq 16,\ldots).
\]

Let the set $S_1$ consist of $1$ place of degree $6$, $16$ places of degree $8$,$12$ places of degree $10$, and $1$ place of degree $12$. Let $\mathfrak p_6$ be the other place of degree $6$, and put \(
\mathfrak{m}
=
\sum_{\mathfrak{p}\in S_1}2\mathfrak{p}
+3\mathfrak p_6.
\) For $|T|=192$, the function field $K=k_{\mathfrak m}^{T}$ has an infinite class field tower by \cref{thm:maincriterion}. Hence, applying \cref{cor:average-conductor}, we obtain
\begin{align*}
    A(2)
    &\geq
\frac{192}{
334-1
+2\cdot 6\left(1-\frac{1}{2^6}\right)
+16\cdot 8\left(1-\frac{1}{2^8}\right)
+12\cdot 10\left(1-\frac{1}{2^{10}}\right)
+1\cdot 12\left(1-\frac{1}{2^{12}}\right)}\\
&\geq 0.317779\cdots.
\end{align*} 
\end{proof}

\begin{remark}
\label{remark:A2}
In this situation, allowing slightly larger ramification depth can still give families of curves with many rational points. A detailed comparison is given in \cref{subsec:6.2}.
\end{remark}

\begin{theorem}\label{thm:6.2}
The Ihara constant $A(3)$ is lower bounded by
      $$ A(3)\geq 0.496074\cdots.$$
    \end{theorem}
\begin{proof}
From the proof of \cite[Proposition 5.6]{Iharacompositio}, there exists a function field $k/\F_3$ with genus $g(k)= 601$ and the sequence of number of places of $k$ with degree $d$ can be given by
\[(a_d(k):d\geq1)=(567,0,0,0, 1,0,0,162,1809).
\]
Let the set $S$ consist of $43$ places of degree $8$ and $2$ places of degree $9$.  For $\mathfrak{m}=\sum_{\mathfrak{p}\in S}3\mathfrak{p}$ and $|T| = 567$, the function field $K = k_{\mathfrak{m}}^{T}$ has an infinite class field tower by \cref{thm:maincriterion}. Applying \cref{cor:average-conductor}, we obtain
\[
\begin{aligned}
A(3)
\geq
\frac{567}{
601-1+\frac12\left(
43\cdot8\left(3-\frac1{3^8}-\frac2{3^{16}}\right)
+2\cdot9\left(3-\frac1{3^9}-\frac2{3^{18}}\right)
\right)}
\geq 0.496074\cdots.
\end{aligned}
\]
This completes the proof of \cref{thm:A2}.

\end{proof}

\begin{remark}
\label{remark:A3}
We consider the same function field as that used in \cite{Iharacompositio}. However,   \cref{thm:maincriterion} yields an improved criterion for proving the infinitude of the corresponding class field tower compared to \cite[Theroem 1.1]{Iharacompositio}. Moreover, in this example, the criterion obtained by applying the main theorem of \cite[Theorem 1.1]{Iharacompositio}  does not prove that the associated class field tower is infinite, whereas our criterion does. Moreover, we note a gap in the average conductor estimate of \cite[Proposition 5.6]{Iharacompositio}: the bound $3(1-3^{-f_\gp})$ is not valid in general. Instead, we use the bound \(
3-3^{-f_\gp}-2\cdot3^{-2f_\gp}
\) established in \cref{cor:average-conductor}.
\end{remark}

\subsection{Comparison with the previous methods}\label{subsec:6.2}
Now we recall \cite[Theorem 1.1]{Iharacompositio}. In the above notations, and let $w_\mathfrak{p}=ef_{\mathfrak p}
\left\lfloor
\frac{v_{\mathfrak p}-1}{p}
\right\rfloor$ , the criterion given in \cite{Iharacompositio} for the infinitude of the corresponding class field tower can be stated as follows: there exists a real number $\tau\in(0,1)$ such that   $$P_{G,\tau}^{DM}=1-d\tau+
\left(d+t-1+
\sum_{\mathfrak{p}\in S}\binom{n_\mathfrak{p}+1+w_\mathfrak{p}}{2}\right)\tau^2<0.$$
However, our main tower criterion are formulated by: there exists a real number $\tau\in(0,1)$ such that
$$P_{G,\tau}^{Zass}=1-d\tau+
\left(d+t-1+
\sum_{\mathfrak{p}\in S}\binom{n_\mathfrak{p}}{2}\right)\tau^2+\sum_{\mathfrak{p}\in S}n_\mathfrak{p}\tau^p<0.$$
Thus, \(P_{G,\tau}^{\mathrm{Zass}}<P_{G,\tau}^{\mathrm{DM}}\) for any prime \(p>2\), and also for \(p=2\) whenever there exists \(\mathfrak p\in S\) such that \(w_{\mathfrak p}>0\). Moreover, the gap between these two polynomials becomes more pronounced as \(p\) increases. This comparison explains the strict improvement obtained above : when applied to the same function field, the Zassenhaus-degree criterion establishes the infinitude of the corresponding class field tower, which cannot be established by the criterion of Duursma and Mak. Thus it gives a better lower bound of Ihara constants of small characteristics.

\section{Acknowledgement}

This research was initiated by the third author in October 2024, during his master's studies at Capital Normal University, inspired by a seminar talk he gave on the work of Hajir, Maire and Ramakrishna \cite{HMR2020}. The project was subsequently developed and further refined over the past two years, eventually leading to the present paper. The third author would like to thank Qi Liu for helpful discussions on class field tower problems. The work was supported in part by the National Key Research and Development Program of China under Grants 2022YFA1004900, 2023YFE0123900, and the National Natural Science Foundation of China under Grants 12361141818, 12426302, 12271084 and 12571577.

\section{AI Declaration}
The authors used OpenAI's ChatGPT 6 to improve the language and clarity of expressions in this paper.  ChatGPT 6 also identified a gap in the average conductor estimate in \cite{Iharacompositio}, as discussed in \cref{remark:A3}. Apart from this observation, AI made no substantive mathematical contributions to this work. The authors take full responsibility for the content of the paper.

\end{document}